\documentclass[12pt]{amsart}
\usepackage{geometry}   
\usepackage[colorlinks,citecolor = red, linkcolor=blue,hyperindex]{hyperref}
\usepackage{euscript,eufrak,verbatim, mathrsfs}
\usepackage[psamsfonts]{amssymb}
\usepackage{bbm}
\usepackage{graphicx}
\usepackage{float}
\usepackage{float, tikz}
\usepackage[all, cmtip]{xy}
\usepackage{upref, xcolor, dsfont}
\usepackage{amsfonts,amsmath,amstext,amsbsy, amsopn,amsthm}
\usepackage{enumerate}
\usepackage{url}

\usepackage{mathtools}
\usepackage{bookmark}
 \usepackage{euscript}
\usepackage{helvet}         
\usepackage{courier}        
\usepackage{type1cm}        
\usepackage{multicol}        
\usepackage[bottom]{footmisc}

\newtheorem{theorem}{Theorem}[section]
\newtheorem*{theoremA*}{Theorem A}
\newtheorem*{theoremB*}{Theorem B}
\newtheorem{lemma}[theorem]{Lemma}

\newtheorem{proposition}[theorem]{Proposition}
\newtheorem{corollary}[theorem]{Corollary}

\newtheorem*{definition*}{Definition}

\newtheorem*{observation*}{Observation}

\newtheorem*{assumption*}{Assumption}
\newtheorem*{question*}{Question}
\newtheorem*{problem*}{Problem}

\theoremstyle{remark}

\newtheorem*{remark*}{Remark}

\newcommand{\R}{\mathbb{R}}

\newcommand{\Z}{\mathbb{Z}}

\newcommand{\D}{\mathbb{D}}
\newcommand{\C}{\mathbb{C}}

\newcommand{\T}{\mathbb{T}}
\newcommand{\hh}{\mathbb{H}}

\newcommand{\Hol}{\mathrm{Hol}}
\newcommand{\Harm}{\mathrm{Harm}}

\newcommand{\supp}{\mathrm{supp}}
\newcommand{\sgn}{\mathrm{sgn}}

\newcommand{\ev}{\mathrm{ev}}

\newcommand{\pc}{\mathrm{Poi}}

\newcommand{\hb}{\mathscr{B}}

\newcommand{\an}{\text{\, and \,}}

\newcommand{\where}{\text{\, where \,}}

\numberwithin{equation}{section}

\begin{document}

\title[Weighted Da Lio-Rivi\`{e}re-Wettstein inequality]{Da Lio-Rivi\`{e}re-Wettstein-type inequality for weighted Bergman spaces}

\author
{Yong Han}
\address
{Yong HAN: College of Mathematics and Statistics, Shenzhen University, Shenzhen 518060, Guangdong, China}
\email{hanyong@szu.edu.cn}

\author
	{Yanqi Qiu}
	\address
	{Yanqi QIU: School of Fundamental Physics and Mathematical Sciences, HIAS, University of Chinese Academy of Sciences, Hangzhou 310024, China}
	\email{yanqi.qiu@hotmail.com, yanqiqiu@ucas.ac.cn}

\author{Zipeng Wang}
\address{Zipeng WANG: College of Mathematics and Statistics, Chongqing University, Chongqing
401331, China}
\email{zipengwang2012@gmail.com}

\thanks{YH is supported by the grant NSFC 12201419, NSFC 12131016.  ZW is supported by NSFC 12471116.}

\begin{abstract}
In this paper, inspired by the work of Da Lio-Rivi\`{e}re-Wettstein, we investigate the boundary-value characterizations of weighted Bergman spaces and  establish a weighted   Da Lio-Rivi\`{e}re-Wettstein inequality.   In addition, we obtain analogous results on the upper plane which does not seem to be a direct consequence of the ones on the unit disk.
\end{abstract}


\maketitle

\setcounter{equation}{0}

\section{Introduction}

This research  is  inspired by Da Lio, Rivi\`{e}re and Wettstein's work \cite{Da-2021} on a fractial version  Bourgain-Brezis type inequalities. Their work can be reformulated as the following equality (and norm-equivalence) of function spaces:
\begin{align}\label{DRW-eq}
(B^2(\D)  + h^1(\D))  \cap \Hol(\D)= A^2(\D),
\end{align}
where $\Hol(\D)$ denotes the space of holomorphic functions on the unit disk $\D$ and the precise definitions of $B^2(\D)$ and $h^1(\D)$ are given in \S \ref{sec-main-results}.

In this paper, by generalizing the equality \eqref{DRW-eq} to radial-weighted Bergman spaces, we  prove that the radial-weighted Da Lio-Rivi\`{e}re-Wettstein equality \eqref{DRW-eq} holds if and only if the measure is $(1,2)$-Carleson.

\subsection{Main results}\label{sec-main-results}

 The harmonic Hardy space $h^1(\D)$ is defined by
\[
h^1(\D):= \Big\{ u \in \Harm(\D)\Big| \| u\|_{h^1(\D)} = \sup_{0<r<1}  \frac{1}{2\pi}\int_{0}^{2\pi} |u(r e^{i\theta})| d\theta <\infty \Big\},
\]
where $\Harm(\D)$ denotes the space of all harmonic functions on $\D$.
The Hardy space $H^1(\D)$ is then defined as $H^1(\D) : = h^1(\D)\cap \Hol(\D)$.

Throughout the paper, all measures are assumed to be positive measures.  Given a  measure $\mu$ on $\D$, the associated  weighted harmonic Bergman space $B^2(\D, \mu)$ and weighted Bergman space $A^2(\D, \mu)$ are defined as
\[
B^2(\D, \mu):  = L^2(\D, \mu)\cap \Harm(\D) \an  A^2(\D, \mu): = L^2(\D, \mu)\cap \Hol(\D),
\]
both of which inherit the norm of $L^2(\D,\mu)$.   If $\mu $ is the Lebesgue measure on $\D$, then we use the simplified notation $B^2(\D)$ and $A^2(\D)$.

Let  $\mu$ be a measure on $\D$. Then $\mu$  is called a  $(1,2)$-Carleson measure if  there exists a constant $C>0$ such that
\begin{align}\label{def-onetwo}
\Big(\int_\D |f(z)|^2 \mu(dz) \Big)^{1/2}\le C \| f\|_{H^1(\D)}, \quad \forall f\in H^1(\D).
\end{align}

We say that $\mu$ is  {\it boundary-touching} if its support is not relatively compact in $\D$.  
\begin{theorem}\label{thm-disk-stable}
Let $\mu$ be a  radial boundary-touching  measure on $\mathbb{D}$.   Then  
\[
(B^2(\D,\mu) +  h^1(\D))\cap \Hol(\D)= A^2(\D, \mu)
\]
 if and only if  $\mu$ is a $(1,2)$-Carleson measure.
\end{theorem}

For a  finite measure $\mu$ on $\D$ which is not necessarily radial, a simple situation (which in general is  rather different from the situation in Theorem \ref{thm-disk-stable}, see Corollary \ref{prop-complement} below for more details) for the holomorphic stability of $(B^2(\D,\mu), h^1(\D))$ is provided as follows.  Consider the  linear map  $\mathcal{Q}_{+}$  defined  on the space $\Harm(\D)$  by
\begin{align}\label{def-Q-proj}
\mathcal{Q}_{+} \Big(\sum_{n\ge 0} a_n z^n +  \sum_{n\ge 1} b_n \bar{z}^n \Big) := \sum_{n\ge 0}  a_n z^n.
\end{align}
Then the equality
$
(B^2(\D, \mu) + h^1(\D)) \cap \Hol(\D)= A^2(\D,\mu)
$  holds if  both
\begin{align}\label{h-to-A}
\mathcal{Q}_{+}: h^1(\D)\longrightarrow A^2(\D, \mu)
\end{align}
and
\begin{align}\label{B-to-A}
\mathcal{Q}_{+}: B^2(\D, \mu)\longrightarrow A^2(\D,\mu)
\end{align}
are bounded linear operators.

It is not hard to see that the operator  \eqref{h-to-A} is bounded if and only if
\begin{align}\label{k-theta-Berg}
\sup_{\theta \in [0, 2\pi)} \int_\D \frac{\mu(dz)}{|1- e^{-i\theta} z|^2}<\infty.
\end{align}
Hence the boundedness of the operator \eqref{h-to-A}   in general fails even for a radial $(1,2)$-Carleson measure.  On the other hand, the boundedness of \eqref{B-to-A} holds for all radial finite measure $\mu$ on $\D$. For general weight, a  sufficient condition for the boundedness of the operator  \eqref{B-to-A} is that the {\it Bergman projection} being bounded on $L^2(\D, \mu)$ (which then is equivalent to the condition that   $\mu$ is B\'ekoll\'e-Bonami's $B_2$-weight, see \cite{B-B} for more details on Bergman projections). Consequently, for any $B_2$-weight $\mu$ on $\D$ satisfying \eqref{k-theta-Berg}, we have  
\[
(B^2(\D, \mu)  +  h^1(\D))\cap \Hol(\D) =   A^2(\D, \mu). 
\]

 For a radial boundary-touching finite  measure $\mu$ on $\D$, the space $B^2(\D,\mu)$ is complete and  $B^2(\D,\mu)+h^1(\D)$ is a Banach space when equipped with the norm:
\[
\| f\|_{B^2(\D,\mu)+h^1(\D)}: = \inf\Big\{\| g \|_{B^2(\D,\mu)}  + \| h\|_{h^1(\D)}\Big| f= g + h, \,\, g \in B^2(\D,\mu) \an h \in h^1(\D) \Big\}.
\] Recall that a closed subspace $B_1$ of a Banach space $B$ is called {\it complemented} in $B$  if there exists a bounded linear projection from $B$ onto $B_1$.

\begin{corollary}\label{prop-complement}
Let $\mu$ be a radial  boundary-touching $(1,2)$-Carleson measure  on $\D$. Then
\[
(B^2(\D,\mu)+h^1(\D)) \cap \Hol(\D)
\]
 is a closed subspace of $B^2(\D,\mu)+h^1(\D)$. Moreover, the above subspace  is complemented in $B^2(\D,\mu)+h^1(\D)$ if and only if $\mu$ satisfies the condition \begin{align}\label{conv-sing-int}
\int_\D\frac{\mu(dz)}{1-|z|^2}<\infty.
\end{align}
\end{corollary}

Let $\hh = \{z\in \C| \Im (z)>0\}$ denote the upper half plane. In this case, we study the harmonic Zen-type spaces (which  reduce to the usual harmonic Bergman space when the weight is the Lebesgue measure on $\hh$),  see \cite{zen-2009,Pott-jfa}.

A measure $\mu$ on $\hh$ is called {\it boundary-touching} if its support is not contained in
\begin{align}\label{eqn-hh-varepsilon}
\hh_{\varepsilon}:  = \{z\in \C| \Im(z)>\varepsilon\}
\end{align}
for any $\varepsilon>0$. 
Given a  horizontal translation-invariant  boundary-touching measure $\mu$  on $\hh$,  define the harmonic Zen-type space by
\begin{align}\label{def-hb}
\hb^2(\hh, \mu): = \Big\{g \in \Harm(\hh)\Big| \| g\|_{\hb^2(\hh, \mu)}  = \sup_{L>0} \Big(\int_\hh | g(z+iL)|^2 \mu(dz)\Big)^{1/2} <\infty \Big\},
\end{align}
where $\Harm(\hh)$ denotes the set of all harmonic functions on $\hh$.   For a  horizontal translation-invariant  boundary-touching measure $\mu$ on $\hh$, the space $\hb^2(\hh,\mu)$ is always complete and thus is a Hilbert space.
   Recall that the Poisson kernel for $\mathbb{H}$ at the point $z = x+i y \in \hh$ is given by
\begin{align}\label{H-Poi-kernel}
P_{z}^{\mathbb{H}}(t):=\frac{1}{\pi}\frac{y}{(x-t)^2+y^2}, \quad t \in \R.
\end{align}
For any $g \in \Harm(\hh)$ and  any $y>0$, define $g_y: \R\rightarrow \C$ by
\begin{align}\label{def-gline}
g_y(x) : =  g(x+iy), \, \forall \, x\in \R.
\end{align}
  Set
\begin{align}\label{def-pc}
\pc(\hh): = \Big\{g\in \Harm(\hh)\Big|   g_y \in L^2(\R) \an g_{y +y'} = P_{iy'}^\hh* g_{y} \, \,  \forall y, y'>0\Big\}.
\end{align}
We given in the following the relation between the Zen-type space $\hb^2(\hh, \mu)$ and the usual harmonic Bergman space
\[
B^2(\hh, \mu): = L^2(\hh,\mu) \cap \Harm(\hh).
\]
For any horizontal translation-invariant boundary-touching measure $\mu$ on $\hh$,
\begin{align}\label{hb-and-b}
\hb^2(\hh, \mu) = B^2(\hh, \mu)  \cap \pc(\hh).
\end{align}
The proof of \eqref{hb-and-b} is given in  Proposition \ref{prop-zen} in the Appendix. The harmonic Hardy space $h^1(\hh)$ is defined by
\[
h^1(\hh):= \Big\{ u \in \Harm(\hh)\Big| \| u\|_{h^1(\hh)} = \sup_{y>0}  \int_{\R} |u(x + i y)| dx<\infty \Big\}.
\]
The Hardy space $H^1(\hh)$ is then defined as $H^1(\hh) : = h^1(\hh)\cap \Hol(\hh)$. A measure $\mu$ on $\hh$ is called a  $(1,2)$-Carleson measure if  there exists a constant $C>0$ such that
\begin{align}\label{def-onetwoup}
\Big(\int_\hh |f(z) |^2 \mu(dz)\Big)^{1/2} \le C \| f\|_{H^1(\hh)}, \quad \forall f\in H^1(\hh).
\end{align}

\begin{theorem}\label{thm-upper-stable}
Let $\mu$ be a horizontal translation-invariant boundary-touching measure on $\hh$. Then
\[
(\hb^2(\hh, \mu) +  h^1(\hh))\cap \Hol(\hh) = \hb^2(\hh, \mu)\cap \Hol(\hh) 
\]
 if and only if $\mu$ is a $(1,2)$-Carleson measure.
\end{theorem}





\subsection{Sketch of the proof}
Let us give a sketch of the proof of Theorem \ref{thm-disk-stable}.  Our proof of Theorem \ref{thm-disk-stable} is based on a generalization of  the following one-dimensional Bourgain-Brezis-type inequality proved by Da Lio-Rivi\`{e}re-Wettstein:  there exists a universal constant $C>0$ such that for any smooth function $u\in C^\infty(\T)$ with $\int u(e^{i\theta}) d\theta = 0$,
\begin{align}\label{BB-ineq}
\|u\|_{L^2(\mathbb{T})}\leq C\Big(\|(-\Delta)^{1/4}u\|_{H^{-1/2}(\mathbb{T})+L^1(\mathbb{T})}+
\|\mathcal{H}(-\Delta)^{1/4}u\|_{H^{-1/2}
(\mathbb{T})+L^1(\mathbb{T})}\Big),
\end{align}
where the Hilbert transform of $u$ is given by
\[
\mathcal{H}u  (e^{i\theta})= \sum_{n\in \Z }\text{sgn}(n)\widehat{u}(n) e^{in \theta}
\]
and the $1/4$-fractional Laplace transform $(-\Delta)^{1/4}u$ is given by
\[
(-\Delta)^{1/4}u (e^{i\theta})  =  \sum_{n\in \Z} |n|^{1/2}\widehat{u}(n) e^{i n \theta}.
\]
The inequality \eqref{BB-ineq} implies
\begin{align}\label{BHB-ineq}
\|f\|_{B^2(\D)}\leq C(\|f\|_{B^2(\D)+h^1(\D)}+
\|\mathcal{H}f\|_{B^2(\D)+h^1(\D)}), \quad \forall f\in \Harm(\D)\cap L^\infty(\D),
\end{align}
where,  slightly by abusing the notation, $\mathcal{H}f$ is defined by
\begin{align}\label{def-Hil}
\mathcal{H} f (z)  =  \sum_{n\ge 1} a_n z^n - \sum_{n\ge 1} b_n \bar{z}^n, \quad \text{provided that\,\,} f (z)=   \sum_{n\ge 0} a_n z^n + \sum_{n\ge 1} b_n \bar{z}^n
\end{align}

In our situation, we shall show that, if $\mu$ is a radial boundary-touching measure, then  $\mu$ is a $(1,2)$-Carleson measure on $\D$ if and only if there exists a constant $C_\mu>0$ such that  for any bounded harmonic function $f \in \Harm(\D)\cap L^\infty(\D)$,
\begin{align}\label{weight-BHB}
\| f\|_{B^2(\D, \mu)} \le C_\mu(\| f\|_{B^2(\D, \mu)+h^1(\D)} + \| \mathcal{H} f\|_{B^2(\D,\mu)+h^1(\D)}).
\end{align}
 The inequality \eqref{weight-BHB} applied to holomorphic functions immediately gives the result stated in Theorem \ref{thm-disk-stable}.

 The proof of the inequality \eqref{weight-BHB}
relies on a weighted version of Bourgain-Brezis-type inequality obtained in Theorem \ref{thm-BBB-weight} below for a Fourier multiplier operator associated to the moment sequence of the radial weight $\mu$. More precisely, define a Fourier multiplier operator by
\begin{align}\label{def-A-mu}
\mathcal{A}_\mu  u \sim
\sum_{n\in \Z}\Big(\int_{\D}|z|^{2|n|}\mu(dz)\Big)^{-1/2}\widehat{u}(n) e^{in\theta}, \quad u \in C^\infty(\T),
\end{align}
where
\[
\widehat{u}(n): =\frac{1}{2\pi} \int_0^{2\pi} u(e^{i\theta})e^{-in\theta}d\theta, \quad n \in \Z.
\]
Define  also a Sobolev-type space corresponding to the radial weight $\mu$ by
\begin{align}\label{def-H-mu}
H_\mu(\T)  := \Big\{v\sim \sum_{n\in \Z} \widehat{v}(n) e^{i n \theta}\Big|\|v\|_{H_\mu(\T)}=\Big(\sum_{n\in\mathbb{Z}} |\widehat{v}(n)|^2\int_\D |z|^{2|n|} \mu(dz)\Big)^{1/2}<\infty\Big\}.
\end{align}
  For a general $\mu$, the coefficients of  a formal Fourier series $v\in H_\mu(\T)$ may have non-polynomial growth and it may  not represent a distribution in $\mathcal{D}'(\T)$. However, if $\mu$ is the Lebesgue measure on $\D$, then $H_\mu(\T)$ coincides with the Sobolev space $H^{-1/2}(\T) \subset \mathcal{D}'(\T)$.

\begin{theorem}\label{thm-BBB-weight}
Let $\mu$ be a radial boundary-touching finite measure on $\mathbb{D}$. Then $\mu$ is a $(1,2)$-Carleson measure if and only if there exists a universal constant $C_\mu>0$ such that  for any smooth function $u\in C^\infty(\T)$,
\begin{align}\label{A-H-ineq}
\|u\|_{L^2(\mathbb{T})}\leq C_\mu \Big(\|\mathcal{A}_\mu u\|_{H_\mu(\mathbb{T})+L^1(\mathbb{T})}+
\|\mathcal{H} \mathcal{A}_\mu u\|_{H_{\mu}
(\mathbb{T})+L^1(\mathbb{T})}\Big).
\end{align}
\end{theorem}

\section{Preliminaries on Carleson measures}
Recall that throughout the paper, all measures are assumed to be positive measures. We shall use the famous geometric characterization of  the $(1,2)$-Carleson measures on $\D$ defined in \eqref{def-onetwo}.
For any interval $I\subset\T$, the Carleson box $S_I$ is defined by
\[
S_I=\Big\{z\in\mathbb{D}\Big|\frac{z}{|z|}\in I, 1-\frac{|I|}{2\pi} \leq |z|<1\Big\},
\]
where $|I|$ denotes the arc-length of  $I$. Let $|S_I|$ denote the Lebesgue measure of $S_I$, then a measure  $\mu$ on $\D$ is a $(1,2)$-Carleson measure  if and only if  (see \cite{carleson-1962} and  \cite{duren-1969})
\begin{align}\label{ch-carleson}
\sup_{\text{$I$ is an arc in $\T$}}\frac{\mu(S_I)}{|S_I|}<\infty.
\end{align}
In particular, when $\mu$ is a radial measure on $\D$,  the condition \eqref{ch-carleson}  is  reduced to the following simpler one which can be verified directly.

\begin{lemma}\label{lem-radial-Carleson}
Let $\mu(dz) = \sigma(dr)d\theta$ be a radial measure on $\D$. Then $\mu$ is a $(1,2)$-Carleson measure   if and only if
\begin{align}\label{E:C-B}
\sup_{0<\delta<1} \frac{\sigma([1-\delta,1))}{\delta}<\infty.
\end{align}
\end{lemma}

The $(1,2)$-Carleson measures on the upper half plane $\hh$ is defined in \eqref{def-onetwoup} and its geometric characterization  (see, e.g., \cite[Thm. 2.1]{Rydhe}) is given as follows:  a positive  Radon measure $\mu$ on $\hh$ is a $(1,2)$-Carleson measure on $\hh$ if and only if
\begin{align}\label{ch-carleson-up}
\sup_{\text{$I$ is an interval in $\R$}}\frac{\mu(Q_I)}{|Q_I|}<\infty,
\end{align}
where $|Q_I|$ is the Lebesgue measure of the Carleson box $Q_I$   defined by
\[
Q_I=\Big\{z=x+iy\in \hh \Big| x \in I, 0 <y<|I|\Big\},
\]
here $|I|$ denotes the Lebesgue measure of the interval $I\subset \R$. In particular, we have
\begin{lemma}\label{lem-carleson-up}
Let $\mu(dz)= dx\Pi(dy)$ be a horizontal translation-invariant  measure on $\hh$. Then $\mu$ is a $(1,2)$-Carleson measure if and only if
\begin{align}\label{good-Pi}
\sup_{y>0} \frac{\Pi((0, y])}{y}<\infty.
\end{align}
\end{lemma}

\section{The disk case}

This section is mainly devoted to proving Theorem \ref{thm-disk-stable} and Corollary \ref{prop-complement}. We shall use the following elementary observation: if $\mu(dz)= \sigma(dr)d\theta$ is a radial boundary-touching finite measure on $\D$, then
\begin{itemize}
\item  both  $A^2(\D,\mu)$ and  $B^2(\D,\mu)$ are closed in $L^2(\D,\mu)$;
\item for any $z\in \D$, the evaluation map $\mathrm{ev}_z:  B^2(\D,\mu)+h^1(\D) \longrightarrow \C$ defined by
\begin{align}\label{def-ev}
\mathrm{ev}_z(f)=  f(z)
\end{align}
is a continuous linear functional on $B^2(\D,\mu)+h^1(\D)$;
\item for any $\rho \in (0,1)$,
\begin{align}\label{low-sigma}
\sigma_k= \int_0^1 r^{2k} \sigma(dr)  \ge\int_{[\rho, 1)} r^{2k} \sigma(dr) \ge  \rho^{2k} \sigma([\rho, 1)).
\end{align}
\end{itemize}

Recall the definition \eqref{def-H-mu} of the space $H_\mu(\T)$: for any $v\in H_\mu(\T)$, we set
\begin{align}\label{Hmu-sum}
\|v\|_{H_\mu(\T)}^2 = \sum_{n\in\mathbb{Z}} |\widehat{v}(n)|^2\int_\D |z|^{2|n|} \mu(dz) =  2 \pi \sum_{n\in\mathbb{Z}} |\widehat{v}(n)|^2 \sigma_n.
\end{align}
By \eqref{low-sigma} and \eqref{Hmu-sum},  for any $r \in [0,1)$, the Poisson transformation $P_r^{\D}*v$ of $v\in H_\mu(\T)$ is a smooth function  given by
\[
P_r^{\D}*v (e^{i\theta}) = \sum_{n\in \Z} r^{|n|}  \widehat{v}(n) e^{in\theta}.
\]
In particular, any $v\in H_\mu(\T)$ has a natural harmonic extension (denoted again by $v$) on $\D$:
\begin{align}\label{v-harm-ext}
v(z) = \sum_{n\in \Z} \widehat{v}(n) e_n(z) \quad \text{with}\quad
e_n(z)= \left\{
\begin{array}{cc}
z^n & \text{if $n\ge 0$}
\vspace{2mm}
\\
\bar{z}^{|n|} & \text{if $n\le -1$}
\end{array}
\right.
\end{align}

\subsection{The derivation of Theorem  \ref{thm-disk-stable} from Theorem \ref{thm-BBB-weight}}\label{sec-15-11}
We begin with the following simple observation which will be useful later.
\begin{lemma}\label{lem-hat-less-sum}
Suppose that $\mu$ is a radial boundary-touching $(1,2)$-Carleson measure on $\D$, then there exists a constant $C_\mu$ such that for any $f\in C^\infty(\T)$, one has
\begin{align}
|\widehat{f}(0)|\leq C_\mu \left\|f\right\|_{H_\mu(\T)+L^1(\T)}.
\end{align}
\end{lemma}
\begin{proof}
For any $f\in C^\infty(\T)$, write $f=g+h$ with $g\in H_\mu(\T)$ and $h\in L^1(\T)$. By the definition \eqref{def-H-mu}, one has
\begin{align*}
|\widehat{g}(0)|\leq \frac{\|g\|_{H_\mu(\T)}}{\sqrt{\mu(\D)}}=C_\mu \|g\|_{H_\mu(\T)}
\end{align*}
and
\begin{align*}
|\widehat{h}(0)|=\left|\frac{1}{2\pi}\int_{0}^{2\pi} h(e^{i\theta})d\theta\right|\leq \frac{1}{2\pi} \int_{0}^{2\pi} \left|h(e^{i\theta})\right|d\theta=\left\|h\right\|_{L^1(\T)}.
\end{align*}
Hence
\[
|\widehat{f}(0)|\leq |\widehat{g}(0)|+|\widehat{h}(0)|\leq \left(1+C_\mu\right)\left(\|g\|_{H_\mu(\T)}+\left\|h\right\|_{L^1(\T)}\right).
\]
Since the decomposition $f=g+h$ is arbitrary, one gets
\[
|\widehat{f}(0)|\leq (1+C_\mu)\left\|f\right\|_{H_\mu(\T)+L^1(\T)}.
\]
This completes the proof.
\end{proof}

 \begin{proof}[Proof of Theorem  \ref{thm-disk-stable}]
If $(B^2(\D,\mu) + h^1(\D)) \cap \Hol(\D)= A^2(\D, \mu)$, then we have set-theoretical inclusion  \[
H^1(\D) \subset A^2(\D,\mu).
\]
It follows that $\mu$ is a finite measure, which combined with the assumption of the theorem, implies that  $\mu$ is a  radial boundary-touching finite  measure on $\D$. Therefore, $A^2(\D, \mu)$ is complete. Hence, by the Closed Graph Theorem,  the embedding $H^1(\D)\subset A^2(\D,\mu)$ is continuous. In other words,  $\mu$ is a $(1,2)$-Carleson measure on $\D$.

  Now assume that  $\mu$ is a radial boundary-touching $(1,2)$-Carleson measure $\mu$ on $\D$.  To prove the equality $(B^2(\D,\mu) + h^1(\D)) \cap \Hol(\D)= A^2(\D, \mu)$, it suffices to show that there exists a constant $C>0$ such that
\begin{align}\label{A-mu-BH}
\| f \|_{A^2(\D,\mu)}\le C \| f\|_{B^2(\D,\mu)+h^1(\D)}, \quad \forall f \in \Hol(\D).
\end{align}
Indeed, for any $f\in \Hol(\D)$,  write $f_r(z): = f(rz)$ for all $0<r<1$, since  $\mu$ is radial,
\[
\lim_{r\to 1^{-}} \| f_r\|_{A^2(\D,\mu)} = \| f\|_{A^2(\D,\mu)}.
\]
Clearly, by the contractivity of the Poisson convolution on both $B^2(\D,\mu)$ and $h^1(\D)$,
\[
\limsup_{r\to 1^{-}} \| f_r\|_{B^2(\D,\mu) + h^1(\D)} \le \| f\|_{B^2(\D,\mu) + h^1(\D)}.
\]
Therefore, it suffices to show that \eqref{A-mu-BH} holds for all  $f$ belonging to the following class:
\[
\Hol(\overline{\D}) = \{f|\text{$f$ is holomorphic in a neighborhood of $\overline{\D}$}\}.
\]
We now proceed  to the  derivation of the inequality \eqref{A-mu-BH} for
 any $f\in \Hol(\overline{\D})$ from the inequality  \eqref{A-H-ineq} obtained in Theorem \ref{thm-BBB-weight}. Notice that any $v\in \Harm(\overline{\D})$ can be written as
\[
v(z) = \sum_{n\in \Z} a_n e_n(z),\quad z\in\overline{\D},
\]
where $e_n$ is defined in \eqref{v-harm-ext}.
Then the restriction of $v$ on $\T$ is
\[
v|_\mathbb{T}(e^{i\theta})=\sum_{n\in \Z} a_n e^{in\theta}.
\]
By the radial assumption on $\mu$ and the definition \eqref{def-H-mu},
\begin{align}\label{B-norm}
\| v\|_{B^2(\D,\mu)}^2 = \sum_{n\in \Z}|a_n|^2\int_\D |z|^{2|n|} \mu(dz)=\|v|_\mathbb{T}\|_{H_\mu(\T)}^2.
\end{align}
For any $f\in \Hol(\overline{\D})$, consider its decomposition on the unit disk $\mathbb{D}$:
\[
f(z)=g(z)+h(z), \quad z\in\D,
\]
with $g\in B^2(\mathbb{D},\mu)$ and $h\in h^1(\mathbb{D})$.
Then, it follows that for any $r\in(0,1)$, 
$$
f_r(z)=g_r(z)+h_r(z),\quad z\in\mathbb{D}.
$$
Notice that  $f_r, g_r, h_r\in \Harm(\overline{\D})$,  hence their restrictions on the unit circle are well-defined (even continuous) and we denote:
\[
\widetilde{f}_r=f_r|_\mathbb{T},\quad  \widetilde{g}_r=g_r|_\mathbb{T},\quad
\widetilde{h}_r=h_r|_\mathbb{T}.
\]
Therefore,
\[
\widetilde{f}_r(e^{i\theta})=\widetilde{g}_r(e^{i\theta})+\widetilde{h}_r(e^{i\theta}),\quad e^{i\theta}\in \T.
\]
Since $g_r\in \Harm(\overline{\D})$,  we obtain
$\widetilde{g}_r\in H_\mu(\mathbb{T})$ and
\[
\|\widetilde{g}_r\|_{H_\mu(\mathbb{T})}=\|g_r\|_{B^2(\mathbb{D},\mu)}\leq
\|g\|_{B^2(\mathbb{D},\mu)}.
\]
Combining with the following standard fact
\[
\|\widetilde{h}_r\|_{L^1(\T)}=\frac{1}{2\pi}\int_0^{2\pi}
|\widetilde{h}_r(e^{i\theta})|d\theta=
\frac{1}{2\pi}\int_0^{2\pi}
|h(re^{i\theta})|d\theta\leq \|h\|_{h^1(\D)},
\]
we get
\begin{align}\label{eqn-hat-f-r}
\|\widetilde{f}_r\|_{H_\mu(\T)+L^1(\T)}\leq \|\widetilde{g}_r\|_{H_\mu(\mathbb{T})}+\|\widetilde{h}_r\|_{L^1(\T)}\leq \|g\|_{B^2(\mathbb{D},\mu)}+\|h\|_{h^1(\D)}.
\end{align}
Now by taking $u = \widetilde{f}_r \in C^\infty(\T)$ in \eqref{A-H-ineq}, we can conclude that there is a constant $C_\mu$ depending only on the measure $\mu$ (in the following $C_\mu$ will denote a constant depending on $\mu$ which can be different in different places)  such that for all $r\in(0,1)$ and $f\in \Hol(\overline{\D})$, the following inequality holds:
\begin{align}\label{eqn-A-hat-fr}
\|\mathcal{A}_\mu^{-1}(\widetilde{f}_r)\|_{L^2(\T)}\le C_\mu (\| \widetilde{f}_r\|_{H_\mu(\mathbb{T})+L^1(\mathbb{T})}+
\|\mathcal{H}\widetilde{f}_r \|_{H_\mu (\mathbb{T})+L^1(\T)}),
\end{align}
with $\mathcal{A}_\mu$  defined in \eqref{def-A-mu} and  $\mathcal{H}f$ defined in \eqref{def-Hil}.
Since $f\in \Harm(\overline{\D})$, we have
\[
\mathcal{H}\widetilde{f}_r=\widetilde{f}_r-f_r(0)=\widetilde{f}_r-f(0).
\]
By Lemma \ref{lem-hat-less-sum}, we have
\[
|f(0)|=|\widehat{\widetilde{f}}_r(0)|\leq C_\mu \| \widetilde{f}_r\|_{H_\mu(\mathbb{T})+L^1(\mathbb{T})}.
\]
Hence by  \eqref{eqn-hat-f-r} and \eqref{eqn-A-hat-fr}, 
\[
\|\mathcal{A}_\mu^{-1}(\widetilde{f}_r)\|_{L^2(\T)}\le  C_\mu  \| \widetilde{f}_r\|_{H_\mu(\T)+L^1(\T)}\leq C_\mu\left( \|g\|_{B^2(\mathbb{D},\mu)}+\|h\|_{h^1(\D)}\right).
\]
In other words,   the following inequality holds (we emphasize that the term on the right-hand side is independent of the choice $r$):
\[
\|\mathcal{A}_\mu^{-1}(\widetilde{f}_r)\|_{L^2(\T)}\le  C_\mu
\left(\|g\|_{B^2(\mathbb{D},\mu)}+\|h\|_{h^1(\D)}\right).
\]
For any $f = \sum_{n\ge 0} c_nz^n\in \Hol(\overline{\D})$,  we have 
$
\widetilde{f}_r(e^{i\theta})= \sum_{n\ge 0} c_n r^ne^{in\theta}
$
and therefore
\begin{align*}
\|\mathcal{A}_\mu^{-1}(\widetilde{f}_r)\|_{L^2(\T)}^2 = \sum_{n\ge 0} |c_n|^2 r^{2n}\int_{\D}|z|^{2|n|}\mu(dz).
\end{align*}
It follows that
\begin{align}
\|f\|_{A^2(\D,\mu)}=\lim_{r\to 1^- } \bigg[\sum_{n\ge 0} |c_n|^2 r^{2n}\int_{\D}|z|^{2|n|}\mu(dz)\bigg]^{1/2}\leq
  C_\mu
\left(\|g\|_{B^2(\mathbb{D},\mu)}+\|h\|_{h^1(\D)}\right).
\end{align}
Since the decomposition  $f=g+h$ with $g\in B^2(\mathbb{D},\mu),h \in h^1(\mathbb{D})$ is arbitrary, we have
\[
\|f\|_{A^2(\D,\mu)}\leq C_\mu\|f\|_{B^2(\mathbb{D},\mu)+h^1(\mathbb{D})}.
\]

Thus we obtain the desired inequality  \eqref{A-mu-BH} for all $f \in \Hol(\overline{\D})$  and complete the derivation of Theorem  \ref{thm-disk-stable} from Theorem \ref{thm-BBB-weight}.
\end{proof}

\subsection{The proof of Theorem \ref{thm-BBB-weight} }
 We say that  a pair $(a, b)$ of sequences $a=(a(n))_{n\in\mathbb{Z}}$ and $b=(b(n))_{n\in\mathbb{Z}}$  is $\mu$-adapted if  the following conditions are satisfied:
\begin{itemize}
\item[(i)] $a(0)\ne 0$;
\item[(ii)] for any $n\in\mathbb{Z}^* = \Z\setminus \{0\}$,
\begin{align}\label{ab-cond}
|a(n)|^2 b(n) \sgn(n) =  \sigma_n^{-1}, \quad \text{where\,\,} \sigma_n: = \int_0^1 r^{2|n|}\sigma(dr)>0;
\end{align}
\item[(iii)] there exists a constant $C_b$ such that
\begin{align}\label{def-Cb}
0<\frac{1}{C_b}\leq |b(n)|\leq C_b,\quad n\in\mathbb{Z}^*.
\end{align}
\end{itemize}
Note that the condition \eqref{ab-cond} implies in particular that $b(n)\in \R$ for all $n\in \Z^*$.

\begin{proposition}\label{thm:disk-2B}
Suppose $\mu$ is a radial boundary-touching $(1,2)$-Carleson measure on $\mathbb{D}$. Let $(a,b)$ be a $\mu$-adapted pair of sequences. Then there exists a constant $C$ such that
\begin{align}\label{2B-ineq-goal}
\|u\|_{L^2(\mathbb{T})}\leq C(\|\mathcal{T}_a u\|_{H_\mu(\mathbb{T})+L^1(\mathbb{T})}
+\|\mathcal{T}_b\mathcal{T}_au\|_{H_\mu(\mathbb{T})+L^1(\T)}), \quad \forall u\in C^\infty(\mathbb{T}),
\end{align}
where $\mathcal{T}_a$ and $\mathcal{T}_b$ are the Fourier multipliers   defined by
\[
\widehat{\mathcal{T}_a u}(n)=a(n) \widehat{u}(n) \an
\widehat{\mathcal{T}_b u}(n)=b(n)\widehat{u}(n),\quad n\in\mathbb{Z}.
\]
\end{proposition}

The next criterion of radial $(1,2)$-Carleson measures will be useful for us.

\begin{lemma}\label{prop:bergman-carleson}
Let $\alpha (dr)$ be a finite measure on $[0,1)$.  Then the inequality
\begin{align}\label{sup-theta}
\sup_{\theta\in [0,2\pi)} \Big|\int_0^1 \frac{\sin\theta}{(r-\cos\theta)^2+\sin^2\theta} \alpha(dr)\Big|<\infty
\end{align}
holds if and only if
\[
\sup_{0<\delta<1} \frac{\alpha([1-\delta, 1))}{\delta}<\infty.
\]
\end{lemma}

We postpone the proof of Lemma \ref{prop:bergman-carleson} for a while and proceed to the proof of Theorem~\ref{thm-BBB-weight}.

\begin{lemma}\label{lem-bdd-w}
Let $\mu(dz)= \sigma(dr)d\theta$ be a radial boundary-touching $(1,2)$-Carleson measure on $\D$, then there exists a function $w_\sigma \in L^\infty(\T)$ such that
\begin{align}\label{w-fourier}
\widehat{w}_\sigma (n)=\sgn(n)\sigma_n,\quad n\in\mathbb{Z}.
\end{align}
\end{lemma}
\begin{proof}
Under the assumption of the lemma,  set
\begin{align}\label{w-def}
w_\sigma(e^{i\theta})= 2i\int_0^1\frac{r^2\sin\theta}{|r^2-e^{-i\theta}|^2}\sigma(dr).
\end{align}
We first show that $w_\sigma\in L^\infty(\T)$.  Indeed,  by change-of-variable $r = \sqrt{s}$,
\[
w_\sigma(e^{i\theta}) =  2 i
\int_0^1   \frac{\sin\theta}{(\cos\theta-s)^2+\sin^2\theta}\sigma'(ds),
\]
where $\sigma'(ds) = s \sigma_{*}(ds)$ with
$\sigma_*(ds)$ being  the push-forward of  the measure $\sigma$ under the  map $s=r^2$. By Lemma \ref{lem-radial-Carleson}, there exists a constant $C>0$ such that
\[
\sigma([1-\delta, 1))\le C \delta, \quad \forall \delta\in (0,1).
\]
Then, by the definition of $\sigma'$, there exists a constant $C'>0$ such that
\[
\sigma'([1-\delta, 1))\le C' \delta, \quad \forall \delta\in (0,1).
\]
Hence, $w_\sigma \in L^\infty(\mathbb{T})$ by Lemma \ref{prop:bergman-carleson}.

It remains to show the equality \eqref{w-fourier}.
Since $w_\sigma\in L^\infty(\T)$, we have
\[
\sup_{\theta\in [0,2\pi)}\int_0^1\frac{r^2|\sin\theta|}{|r^2-e^{-i\theta}|^2}\sigma(dr)  = \sup_{\theta\in [0,2\pi)} \Big| \int_0^1\frac{r^2\sin\theta}{|r^2-e^{-i\theta}|^2}\sigma(dr)\Big|<\infty.
\] Therefore, for any $n\in \Z$, by Fubini's Theorem,
\begin{align*}
\widehat{w}_\sigma (n)& = \frac{1}{2\pi} \int_0^{2\pi}
 \Big(2i\int_0^1\frac{r^2\sin\theta}{|r^2-e^{-i\theta}|^2}\sigma(dr)\Big) e^{-in \theta} d\theta
\\
& = \int_0^1 \Big(  \frac{1}{2\pi} \int_0^{2\pi}
 2i\frac{r^2\sin\theta}{|r^2-e^{-i\theta}|^2}e^{-in \theta} d\theta \Big) \sigma(dr).
\end{align*}
By the elementary identity (which converges absolutely for any fixed $0\le r<1$)
\[
2i \frac{r^2\sin\theta}{|r^2-e^{-i\theta}|^2} = \sum_{n\in\mathbb{Z}^*} \text{sgn}(n) e^{in\theta} r^{2|n|}, \quad \forall r\in [0,1),
\]
we then have
\[
 \frac{1}{2\pi} \int_0^{2\pi}
 2i\frac{r^2\sin\theta}{|r^2-e^{-i\theta}|^2}e^{-in \theta} d\theta   =  \sgn(n) r^{2|n|}
\]
and hence
\[
\widehat{w}_\sigma (n)  = \sgn(n) \int_0^1 r^{2|n|} \sigma(dr)= \sgn(n) \sigma_n.
\]
This is the desired equality \eqref{w-fourier}.
\end{proof}

\begin{proof}[Proof of Proposition \ref{thm:disk-2B}]
Let $u\in C^\infty(\mathbb{T})$. Note that
$
\| u \|_{L^2(\T)} \le \| u- \widehat{u}(0)\|_{L^2(\T)} + |\widehat{u}(0)|.
$
If $\mathcal{T}_a u = f + g$ with $f\in H_\mu(\T)$ and $g\in L^1(\T)$, then
$
a(0) \widehat{u}(0)= \widehat{f}(0)+ \widehat{g}(0).
$
Moreover,
\[
| a(0) \widehat{u}(0)|\le  |\widehat{f}(0)|+ |\widehat{g}(0)|\le \frac{\| f\|_{H_\mu(\T)}}{\sqrt{\mu(\D)}} + \| g\|_{L^1(\T)}\leq \frac{1+\sqrt{\mu(\D)}}{\sqrt{\mu(\D)}}\left(\| f\|_{H_\mu(\T)}+ \| g\|_{L^1(\T)}\right).
\]
By the assumption  that $a(0)\not=0$, we have
\[
|\widehat{u}(0)|\le \frac{1+\sqrt{\mu(\D)}}{|a(0)|\sqrt{\mu(\D)}}\|\mathcal{T}_a u\|_{H_\mu(\T)+L^1(\T)}.
\]
Therefore, from now on,  we may assume that  $\widehat{u}(0) =0$.  Take any pairs of decompositions
\begin{align}\label{2-dec}
\left\{
  \begin{array}{ll}
    \mathcal{T}_a u&=f_1+g_1, \\
    \mathcal{T}_b\mathcal{T}_a u&=f_2+g_2,
  \end{array}
\right.
\end{align}
with $f_1,f_2\in H_\mu(\mathbb{T})$ and $g_1,g_2\in L^1(\mathbb{T})$. Then for any $0<r<1$, we have (the following Poisson convolutions will be used in the proof of the equality \eqref{Pl-id} below)
\begin{align}\label{P-conv}
\left\{
  \begin{array}{ll}
    P_r^\D * (\mathcal{T}_a u) &= P_r^\D * f_1+P_r^\D * g_1, \\
    P_r^\D * (\mathcal{T}_b \mathcal{T}_a u)&=P_r^\D * f_2+P_r^\D * g_2.
  \end{array}
\right.
\end{align}
  That is,
\begin{align}\label{E:coefficientAB}
\left\{
\begin{array}{ll}
r^{|n|}a(n)\widehat{u}(n)&=r^{|n|}\widehat{f}_1(n)+
r^{|n|}\widehat{g}_1(n), \quad n\in\Z,
\\
r^{|n|}b(n)a(n)\widehat{u}(n)&=r^{|n|}\widehat{f}_2(n)+
r^{|n|}\widehat{g}_2(n), \quad n\in \Z.
\end{array}
\right.
\end{align}
From \eqref{E:coefficientAB}, we have
\begin{align}\label{I+II}
\begin{split}
\|P_r^\D * u\|_{L^2(\mathbb{T})}^2&=\sum_{n\in \Z^*} r^{2|n|}|\widehat{u}(n)|^2\\
&=\underbrace{\sum_{n\in \mathbb{Z}^*}\frac{r^{|n|}}{a(n)}\hat{f}_1(n)
r^{|n|}\overline{\hat{u}(n)}}_{\text{denoted by I}}
+\underbrace{\sum_{n\in\mathbb{Z}^*}\frac{r^{|n|}}{a(n)}
\hat{g}_1(n)r^{|n|}\overline{\hat{u}(n)}}_{\text{denoted by II}}.
\end{split}
\end{align}
By Cauchy-Schwarz's inequality,
\begin{align}\label{I-ineq}
\begin{split}
|\text{I}|&\leq \Big(\sum_{n\in\mathbb{Z}^*}r^{2|n|}|
\overline{\hat{u}(n)}|^2\Big)^{1/2}
\Big(\sum_{n\in\mathbb{Z}^*}
\frac{r^{2|n|}|\widehat{f}_1(n)|^2}{|a(n)|^2}\Big)^{1/2}\\
&\leq \sqrt{C_b/2\pi}\|P_r^\D *u\|_{L^2(\mathbb{T})}\|P_r^\D *f_1\|_{H_\mu(\T)},
\end{split}
\end{align}
where we have used the fact that if $(a,b)$ is a $\mu$-adapted pair of sequences, then  by \eqref{ab-cond} and \eqref{def-Cb},  for any $v\in H_\mu(\T)$,
\begin{align}\label{a-mu-norm}
\Big(\sum_{n\in\mathbb{Z}^*}
\frac{|\widehat{v}(n)|^2}{|a(n)|^2}\Big)^{1/2}= \Big(\sum_{n\in\mathbb{Z}^*}
|\widehat{v}(n)|^2 |b(n)| \sigma_n \Big)^{1/2}\le \sqrt{C_b/2\pi}\| v\|_{H_\mu(\T)}.
\end{align}

By \eqref{E:coefficientAB},
\begin{align}\label{2=3+4}
\text{II}&=\underbrace{\sum_{n\in\mathbb{Z}^*}\frac{r^{|n|}
(a(n)\widehat{u}(n)-\widehat{f}_1(n))}{|a(n)|^2 b(n)}
r^{|n|}\overline{\widehat{f}_2(n)}}_{\text{denoted by III}}+
\underbrace{\sum_{n\in\mathbb{Z}^*}\frac{r^{|n|}\widehat{g}_1(n)
r^{|n|}\overline{\widehat{g}_2(n)}}
{|a(n)|^2 b(n)}}_{\text{denoted by IV}}.
\end{align}
Then by Cauchy-Schwarz's inequality,
\begin{align*}
&|\text{III}| \le \Big|\sum_{n\in\mathbb{Z}^*}\frac{r^{|n|}
\widehat{u}(n)}{\overline{a(n)} b(n)}
r^{|n|}\overline{\widehat{f}_2(n)}\Big| + \Big|\sum_{n\in\mathbb{Z}^*}\frac{r^{|n|}
\widehat{f}_1(n)}{|a(n)|^2 b(n)}
r^{|n|}\overline{\widehat{f}_2(n)}\Big|
\\
& \leq \|P_r^\D * u\|_{L^2(\mathbb{T})}
\Big(\sum_{n\in \mathbb{Z}^*}\Big|\frac{r^{|n|}\widehat{f}_2(n)}
{a(n)b(n)}\Big|^2\Big)^{1/2}
+\Big(\sum_{n\in \mathbb{Z}^*}\frac{r^{2|n|}|\widehat{f}_1(n)|^2}
{|a(n)|^2|b(n)|}\Big)^{1/2}
\Big(\sum_{n\in \mathbb{Z}^*}\frac{r^{2|n|}|\widehat{f}_2(n)|^2}
{|a(n)|^2 |b(n)|}\Big)^{1/2}.
\end{align*}
Using similar inequality as \eqref{a-mu-norm}, under the conditions \eqref{ab-cond} and \eqref{def-Cb}, we have
\begin{align}\label{III-ineq}
|\text{III}|&\leq\sqrt{C_b/2\pi}\|P_r^\D * u\|_{L^2(\mathbb{T})}\|P_r^\D * f_2\|
_{H_\mu(\mathbb{T})}
+\frac{1}{2\pi}\|P_r^\D * f_1\|_{H_\mu(\mathbb{T})}
\|P_r^\D * f_2\|_{H_\mu(\mathbb{T})}.
\end{align}

We now proceed to the estimate of term $\mathrm{IV}$ in the decomposition \eqref{2=3+4}.
Note that
\[
\overline{\widehat{g}}_2(n)=\widehat{\overline{\widetilde{g}}}_2(n), \where \tilde{g}_2(e^{i\theta}):=g_2(e^{-i\theta}).
\]
For any $0< r<1$, set
\begin{align}\label{def-hr}
h_r=(P_r^\D *g_1)*(P_r^\D *\overline{\widetilde{g}}_2).
\end{align}
A priori, we only have $g_1*\overline{\widetilde{g}}_2 \in L^1(\T)$, but $h_r\in L^2(\T)$ for any $0< r<1$.
By \eqref{ab-cond},
\begin{align*}
\text{IV}= \sum_{n\in\mathbb{Z}^*} \sgn(n) \sigma_n  r^{|n|}\widehat{g}_1(n)
r^{|n|}\overline{\widehat{g}_2(n)} = \sum_{n\in\mathbb{Z}^*}\widehat{h}_r(n)\text{sgn}(n)\sigma_n.
\end{align*}
By Lemma \ref{lem-bdd-w}, $w_\sigma \in L^\infty(\T) \subset L^2(\T)$. Then the Plancherel's identity implies
\begin{align}\label{Pl-id}
\text{IV}=\sum_{n\in\mathbb{Z}^*}\widehat{h}_r(n)\widehat{w}_\sigma(n)
= \int_{\mathbb{T}} h_r \bar{w}_\sigma - \int_\mathbb{T} h_r \int_\mathbb{T} \bar{w}_\sigma.
\end{align}
Hence
\begin{align}\label{IV-ineq}
|\text{IV}|\leq 2\|h_r\|_{L^1(\mathbb{T})}\|w_\sigma\|_{L^\infty(\mathbb{T})}\leq 2\|w_\sigma\|_{L^\infty(\mathbb{T})} \|P_r^\D *g_1\|_{L^1(\mathbb{T})} \|P_r^\D *g_2\|_{L^1(\mathbb{T})}.
\end{align}

By  \eqref{I+II}, \eqref{I-ineq}, \eqref{2=3+4}, \eqref{III-ineq} and \eqref{IV-ineq},   there is a constant $C = C(a, b, \mu)$, depending only on $(a, b)$ and the measure $\mu$ but  not on  $r\in (0,1)$,  such that
\begin{align*}
\| P_r^\D *u \|_{L^2(\mathbb{T})}^2 \le &  C\|P_r^\D * u \|_{L^2(\mathbb{T})} \| P_r^\D *f_1\|_{H_\mu(\mathbb{T})} +  C  \|P_r^\D *u\|_{L^2(\mathbb{T})} \| P_r^\D *f_2\|_{H_\mu(\mathbb{T})} +
\\
& +  C \| P_r^\D *f_1\|_{H_\mu(\mathbb{T})} \| P_r^\D *f_2\|_{H_\mu(\mathbb{T})} + C \|P_r^\D *g_1\|_{L^1(\mathbb{T})} \|P_r^\D *g_2\|_{L^1(\mathbb{T})}.
\end{align*}
Therefore, by a standard argument, there is a constant $C' = C'(a, b, \mu)$ such that
\begin{align*}
\| P_r^\D *u \|_{L^2(\mathbb{T})}&\leq C'
 \Big(\| P_r^\D *f_1\|_{H_\mu(\mathbb{T})} +\| P_r^\D * f_2\|_{H_\mu(\mathbb{T})} + \|P_r^\D *g_1\|_{L^1(\mathbb{T})}+ \|P_r^\D *g_2\|_{L^1(\mathbb{T})}\Big)\\
&\leq
C' \Big(\| f_1\|_{H_\mu(\mathbb{T})} + \|g_1\|_{L^1(\mathbb{T})}+\| f_2\|_{H_\mu(\mathbb{T})} + \|g_2\|_{L^1(\mathbb{T})}\Big),
\end{align*}
where the last inequality is due to the contractive property  of the  Poission convolution on  both $H_\mu(\T)$ and $L^1(\T)$. Let $r$ approach to $1$, then
\[
\|u\|_{L^2(\mathbb{T})}\leq
C' \Big(\| f_1\|_{H_\mu(\mathbb{T})} + \|g_1\|_{L^1(\mathbb{T})}+ \| f_2\|_{H_\mu(\mathbb{T})}  + \|g_2\|_{L^1(\mathbb{T})}\Big),
\]
Since the decompositions \eqref{2-dec}  are arbitrary,   we obtain the desired inequality \eqref{2B-ineq-goal}.
\end{proof}

\begin{proof}[Proof of Theorem \ref{thm-BBB-weight}]
If $\mu = \sigma(dr)d \theta$ is a radial boundary-touching  $(1,2)$-Carleson measure on $\D$, then we obtain  the inequality \eqref{A-H-ineq} from Proposition \ref{thm:disk-2B} by taking
\[
a(n)= \sqrt{2\pi}\Big(\int_{\D}|z|^{2|n|}\mu(dz)\Big)^{-1/2} =  \frac{1}{\sqrt{\sigma_n}}  \an b(n) = \sgn(n).
\]

Conversely, if the inequality \eqref{A-H-ineq} holds, then  by the argument in the first two paragraphs of  \S \ref{sec-15-11},  the measure $\mu$ is a $(1,2)$-Carleson measure on $\D$.
\end{proof}

It remains to prove  Lemma \ref{prop:bergman-carleson}. We shall apply the following result due to Garnett about the boundary behavior of Poisson integrals on the upper half plane $\hh$.

\begin{lemma}[{\cite[Theorem 4.2, Chapter I]{garnett}, \cite[pp.210]{Ramey-1988}}]\label{lem-Ramey}
Let $\nu$ be a measure on $\mathbb{R}$ with
$
\int_\mathbb{R}\frac{1}{1+t^2} \nu(dt)<\infty.
$
Then the following two assertions are equivalent:
\begin{itemize}
\item[(1)]
$
\sup_{y>0}\int_{\R}\frac{y}{t^2+y^2}\nu(dt)<\infty
$
\item[(2)]
$
\sup_{L>0}\frac{\nu([-L,L])}{2L}<\infty.
$
\end{itemize}
\end{lemma}
\begin{proof}[Proof of Lemma \ref{prop:bergman-carleson}]
Note that
\begin{align*}
\sup_{\theta\in [0,2\pi)} \Big|\int_0^1 \frac{\sin\theta}{(r-\cos\theta)^2+\sin^2\theta} \alpha(dr)\Big| = \sup_{\theta\in (0,\pi)} \int_0^1 \frac{\sin\theta}{(r-\cos\theta)^2+\sin^2\theta} \alpha(dr).
\end{align*}
For any $\theta\in(0,\pi)$, consider the point $z=e^{i\theta}=\cos\theta+i\sin\theta$ and recall  the Poisson kernel $P_z^\hh$ at the point $z\in \hh$ given in \eqref{H-Poi-kernel}, then
\[
\int_0^1 \frac{\sin\theta}{(r-\cos\theta)^2+\sin^2\theta}\alpha(dr)=
\pi \int_{\mathbb{R}}P^{\mathbb{H}}_{e^{i\theta}}(t)1_{[0,1)}(t)\alpha(dt).
\]
Consider the M\"obius transformation  $\phi$ defined by
$
\phi(z)=\frac{z-1}{z+1}.
$
Then $\phi$ is an automorphism of the upper half plane and
\[
P^\hh_{\phi(z)}(\phi(t)) |\phi'(t)| =  P^\hh_z(t), \quad z \in \hh,\, t \in \R\setminus \{-1\}.
\]
Note that when $\theta$ ranges over $(0, \pi)$, the image $\phi(e^{i\theta})$ ranges over $i \R_{+}$.
Therefore,
\begin{align*}
\sup_{\theta\in (0,\pi)}
\int_{\mathbb{R}}P^\mathbb{H}_{e^{i\theta}}(t)1_{[0,1)}(t)\alpha(dt)&= \sup_{\theta\in (0,\pi)}
\int_{\mathbb{R}} P^\hh_{\phi(e^{i\theta})}(\phi(t)) |\phi'(t)|1_{[0,1)}(t)\alpha(dt)
\\
&  =
\sup_{y>0}\int_0^1 P_{iy}^{\mathbb{H}}(\phi(t))|\phi'(t)|\alpha(dt).
\end{align*}
By change-of-variable $s = \phi(t)$,
\[
\int_0^1 P_{iy}^{\mathbb{H}}(\phi(t))|\phi'(t)|\alpha(dt)  = \int_{-1}^0 P_{iy}^\hh(s)  \frac{(1-s)^2}{2} \alpha\circ \phi^{-1}(ds).
\]
Then
\begin{align*}
\sup_{\theta\in (0,\pi)} \int_0^1 \frac{\sin\theta}{(r-\cos\theta)^2+\sin^2\theta}\alpha(dr)& =   \frac{\pi}{2} \sup_{y>0}\int_\R \frac{y}{y^2 + s^2} \widetilde{\alpha}(ds),
\end{align*}
where
$
\widetilde{\alpha}(ds) =(1-s)^2\mathds{1}_{(-1,0)}(s) \alpha\circ \phi^{-1}(ds)$.
Clearly, $\int_\R \frac{\widetilde{\alpha}(ds)}{1+s^2}<\infty$. Therefore,  by  Lemma~\ref{lem-Ramey},  the inequality \eqref{sup-theta} holds if and only if
\begin{align}\label{sup-LL}
\sup_{L>0}\frac{\widetilde{\alpha}([-L,L])}{L}<\infty.
\end{align}
By the definition of $\widetilde{\alpha}$,  one can  see that $\alpha(I_L)  \le \widetilde{\alpha}([-L, L]) \le 4 \alpha(I_L)$,
where $I_L$ is the open interval \[
I_L: = \Big(\frac{1-\min(L,1)}{1+\min(L,1)},  1\Big) \subset (0,1).
\]
It follows that,  \eqref{sup-LL} holds if and only if
$\sup_{L>0}\frac{\alpha(I_L)}{L}<\infty$,
which in turn is equivalent to
\[
\sup_{0<\delta<1}\frac{\alpha([1-\delta,1))}{\delta}<\infty.
\]
This completes the  whole proof of Lemma \ref{prop:bergman-carleson}.
\end{proof}

\subsection{Proof of Corollary \ref{prop-complement}}
Fix   a  radial boundary-touching $(1,2)$-Carleson measure
$
\mu(dz)= \sigma(dr) d\theta.
$
By Theorem \ref{thm-disk-stable},
\begin{align}\label{stable-conseq}
(B^2(\D,\mu)+h^1(\D))\cap\Hol(\D) = B^2(\D,\mu)\cap \Hol(\D)= A^2(\D,\mu).
\end{align}
By \eqref{A-mu-BH}, there exists a constant $C= C_\mu>0$, such that  for all $f\in \Hol(\D)$,
\begin{align}\label{A=BH}
\frac{1}{C}\| f\|_{A^2(\D,\mu)}\le \| f\|_{B^2(\D,\mu)+h^1(\D)} \le \|f\|_{B^2(\D,\mu)} = \|f \|_{A^2(\D,\mu)}.
\end{align}
That is, the identity map
\[
id: A^2(\D,\mu) \rightarrow (B^2(\D,\mu)+h^1(\D))\cap\Hol(\D) \subset B^2(\D,\mu)+h^1(\D)
\]
is an isomorphic isomorphism. Thus $(B^2(\D,\mu)+h^1(\D))\cap\Hol(\D)$ is a closed subspace of $B^2(\D,\mu)+h^1(\D)$.

Now suppose that the extra condition \eqref{conv-sing-int} is satisfied. We are going to  show that  the closed subspace $(B^2(\D,\mu)+h^1(\D))\cap\Hol(\D)$ is complemented in $B^2(\D,\mu)+h^1(\D)$. Indeed, under the condition \eqref{conv-sing-int}, for any $\theta \in [0, 2 \pi)$, the following holomorphic function
\[
k_\theta(z): = \frac{1}{1 - e^{-i\theta} z} = \sum_{n\ge 0} e^{-in\theta} z^n,\quad z\in \D
\]
 belongs to  $B^2(\D,\mu)+h^1(\D)$ and
\begin{align}\label{cauchy-kernel}
M_\mu:= \sup_{\theta \in [0, 2\pi)} \| k_\theta\|_{B^2(\D,\mu)+h^1(\D)} \le \sup_{\theta \in [0, 2\pi)} \| k_\theta\|_{A^2(\D,\mu)} = \Big(\int_\D \frac{\mu(dz)}{1-|z|^2}\Big)^{1/2}<\infty.
\end{align}
Recall the definition \eqref{def-Q-proj} of $\mathcal{Q}_{+}$.   Clearly,  since $\mu$ is radial, $\mathcal{Q}_{+}$ defines an orthgonal projection from $B^2(\D, \mu)$ onto $A^2(\D,\mu)$. Then
\begin{align}\label{Q-on-B}
\|\mathcal{Q}_{+}(u)\|_{B^2(\D,\mu)+h^1(\D)} \le \| \mathcal{Q}_{+}(u)\|_{B^2(\D,\mu)} \le \|u\|_{B^2(\D,\mu)}, \quad \forall u \in B^2(\D,\mu).
\end{align}
 Note also that if $v= \sum_{n\in \Z} a_n e_n \in h^1(\D)$, that is,
\[
\widetilde{v}: = \sum_{n\in \Z} a_n e^{in \theta} \in L^1(\T) \an \| v\|_{h^1(\D)}   =\| \widetilde{v} \|_{L^1(\T)},
\]
then it is easy to see that
\[
\mathcal{Q}_{+}  (v)=  \mathcal{Q}_{+} \Big(\sum_{n\in \Z} a_n e_n\Big)=  \frac{1}{2\pi} \int_0^{2\pi}  k_\theta \widetilde{v}(e^{i\theta})  d\theta.
\]
Hence, by  \eqref{cauchy-kernel}, we have
\begin{align}\label{h-Q-A}
\begin{split}
\| \mathcal{Q}_{+}(v)\|_{B^2(\D,\mu)+h^1(\D)} & \le \frac{1}{2\pi} \int_0^{2\pi}  \| k_\theta\|_{B^2(\D,\mu)+h^1(\D)} |\widetilde{v}(e^{i\theta})|  d\theta
\\
& \le  M_\mu \| \widetilde{v}\|_{L^1(\T)}   =  M_\mu \| v\|_{h^1(\D)}.
\end{split}
\end{align}
By \eqref{Q-on-B} and \eqref{h-Q-A} and the definition of the norm on $B^2(\D,\mu)+h^1(\D)$,
\[
\|\mathcal{Q}_{+}(f)\|_{B^2(\D,\mu)+h^1(\D)} \le M_\mu \|f\|_{B^2(\D,\mu)+h^1(\D)}, \quad \forall f \in B^2(\D,\mu)+h^1(\D).
\]
It follows that $\mathcal{Q}_{+}$ defines a bounded linear projection from  $B^2(\D,\mu)+h^1(\D)$ onto
\[
(B^2(\D,\mu)+h^1(\D))\cap \Hol(\D).
\]
Hence $(B^2(\D,\mu)+h^1(\D))\cap \Hol(\D)$ is complemented in $B^2(\D,\mu)+h^1(\D)$.

Finally, assume that the condition \eqref{conv-sing-int} is not satisfied. Then
\[
\sum_{n\ge 0}\sigma_n= \sum_{n\ge 0}\int_0^1 r^{2n} \sigma(dr)= \int_0^1 \frac{\sigma(dr)}{1-r^2}   =  \frac{1}{2\pi}\int_\D \frac{\mu(dz)}{1-|z|^2} =\infty.
\]
 Let us show that $(B^2(\D,\mu)+h^1(\D))\cap\Hol(\D)$ is not complemented in $B^2(\D,\mu)+h^1(\D)$. Otherwise, there exists a bounded linear projection operator
\[
P: B^2(\D,\mu)+h^1(\D)\longrightarrow  B^2(\D,\mu)+h^1(\D)
\]
onto the closed subspace $(B^2(\D,\mu)+h^1(\D))\cap\Hol(\D)$. That is,
\begin{itemize}
\item $P\circ P= P$,
\item  $P(f)  = f$ for all $f \in (B^2(\D,\mu)+h^1(\D))\cap\Hol(\D)$,
\item $P(g) \in (B^2(\D,\mu)+h^1(\D))\cap\Hol(\D)$ for all $g\in B^2(\D,\mu)+h^1(\D)$.
\end{itemize}
Since $\mu$ is radial,  for any $\theta \in [0,2\pi)$, the rotation map $\tau_\theta$ defined by $\tau_\theta(f) (z)= f(e^{i\theta}z)$ preserves both the norms of functions in $B^2(\D,\mu)$ and the norms of functions in $h^1(\D)$. Therefore, $\tau_\theta$ also preserves the norms of functions in $B^2(\D,\mu)+h^1(\D)$:
\[
\| \tau_\theta (f)\|_{B^2(\D,\mu)+h^1(\D)} = \| f\|_{B^2(\D,\mu)+h^1(\D)}, \quad \forall f \in B^2(\D,\mu)+h^1(\D).
\]
Consequently, the operator-norm of the  composition operator
\[
P_\theta= \tau_{-\theta} \circ P \circ \tau_\theta: B^2(\D,\mu)+h^1(\D) \longrightarrow B^2(\D,\mu)+h^1(\D)
\] is bounded by that of $P$:
\[
\|P_\theta \| \le \| P\|.
\]
It can be easily checked that $P_\theta$ is also a projection operator from $B^2(\D,\mu)+h^1(\D)$ onto $(B^2(\D,\mu)+h^1(\D))\cap\Hol(\D)$.

Define a bounded linear operator  $\mathcal{P}: B^2(\D,\mu)+h^1(\D)\longrightarrow B^2(\D,\mu)+h^1(\D)$ via the Bochner integral  (see, e.g., \cite[Section V.5]{yosida})
\[
\mathcal{P}: = \frac{1}{2\pi}\int_0^{2\pi} P_\theta d\theta.
\]
Then
\begin{align}\label{bdd-P}
\|\mathcal{P}\|\le  \sup_{\theta\in [0, 2 \pi)} \| P_\theta\| = \| P\|<\infty
\end{align}
and
\begin{align}\label{P-f-Bochner}
\mathcal{P}(f)= \frac{1}{2\pi}\int_0^{2\pi} P_\theta(f) d\theta, \quad \forall f \in  B^2(\D,\mu)+h^1(\D).
\end{align}
Since the evalutation map $\ev_z$ defined in \eqref{def-ev} is  a continuous linear functional on $B^2(\D, \mu)+h^1(\D)$ for any $z\in \D$,
\begin{align}\label{P-f-z}
[\mathcal{P}(f)] (z) = \frac{1}{2\pi}\int_0^{2\pi} [P_\theta(f)](z) d\theta, \quad \forall f \in  B^2(\D,\mu)+h^1(\D).
\end{align}
Note that $P_\theta(f)= f$ for any $f \in (B^2(\D,\mu)+h^1(\D))\cap\Hol(\D)$ and any $\theta \in [0, 2\pi)$, thus
\[
\mathcal{P}(f) = f, \quad \forall f \in (B^2(\D,\mu)+h^1(\D))\cap\Hol(\D).
\]
On the other hand, for any integer $n\ge 1$,
\begin{align}\label{P=Riesz}
\mathcal{P}(e_{-n}) = 0, \quad \forall n \ge 1.
\end{align}
Indeed, for any $\theta \in [0, 2 \pi)$,
\[
(\tau_\theta (e_{-n}) ) (z) = \overline{( e^{i\theta} z)}^n = e^{-i n \theta} \bar{z}^n = e^{-in  \theta} e_{-n}(z).
\]
Thus $P\circ \tau_\theta (e_{-n}) =  e^{-in  \theta} P(e_{-n})$. By \eqref{stable-conseq},
\[
P(e_{-n})\in (B^2(\D,\mu)+h^1(\D))\cap\Hol(\D) = A^2(\D,\mu),
\]
 we can write
\[
P(e_{-n}) (z)= \sum_{k=0}^\infty c^{(n)}_k z^k \in A^2(\D,\mu),
\]
with
\begin{align}\label{square-c-k}
\| P(e_{-n})\|_{A^2(\D,\mu)}^2 = 2 \pi \sum_{k=0}^\infty |c_k^{(n)}|^2 \sigma_k<\infty.
\end{align}
Thus, for all $z\in \D$,
\begin{align*}
P_\theta (e_{-n}) (z)& = [\tau_{-\theta} (e^{-i n\theta} P(e_{-n}))](z) = e^{-i n \theta} [\tau_{-\theta} ( P(e_{-n}))](z)
\\
&  = e^{-in \theta} \sum_{k=0}^\infty c_k^{(n)} (e^{-i\theta} z)^k = \sum_{k=0}^\infty c_k^{(n)} e^{-i (k+n)\theta} z^k,
\end{align*}
where the last series converges absolutely  by the inequalities \eqref{low-sigma}, \eqref{square-c-k} and
\[
\Big(\sum_{k=0}^\infty |c_k^{(n)} z^k|\Big)^2\le  \sum_{k=0}^\infty |c_k^{(n)}|^2 \sigma_k  \sum_{k=0}^\infty \frac{|z|^{2k}}{\sigma_k} \le \sum_{k=0}^\infty |c_k^{(n)}|^2 \sigma_k  \sum_{k=0}^\infty \frac{|z|^{2k}}{\rho^{2k} \sigma([\rho, 1))}, \quad \forall \rho \in (0,1).
\]
Therefore, by \eqref{P-f-z}, for all $z\in \D$,
\begin{align*}
[\mathcal{P}(e_{-n})](z)& =  \int_0^{2\pi} \sum_{k=0}^\infty c_k^{(n)} e^{-i (k+n)\theta} z^k \frac{d\theta}{2\pi}
 =  \sum_{k=0}^\infty   \int_0^{2\pi}  c_k^{(n)} e^{-i (k+n)\theta} z^k \frac{d\theta}{2\pi} =0.
\end{align*}
This is the desired equality \eqref{P=Riesz}.

However, if we take the harmonic extension of the F\'ejer kernel on $\D$:
\[
\mathcal{F}_N(z)= \sum_{j=-N}^N \Big(1 - \frac{|j|}{N}\Big) e_j(z), \quad N\ge 1,
\]
then
\begin{align}\label{Fejer-norm}
\| \mathcal{F}_N\|_{B^2(\D, \mu)+h^1(\D)} \le \| \mathcal{F}_N\|_{h^1(\D)} = \| \mathcal{F}_N|_\T\|_{L^1(\T)} = 1.
\end{align}
Since $\mathcal{P}$ is a projection onto $(B^2(\D, \mu)+h^1(\D))\cap \Hol(\D)$ and satisfies \eqref{P=Riesz}, we have
\[
\mathcal{P}(\mathcal{F}_N) = \sum_{j=0}^N \Big(1 - \frac{j}{N}\Big) e_j
\]
and, by the radial assumption on $\mu$,
\[
\|\mathcal{P}(\mathcal{F}_N)\|_{A^2(\D,\mu)}^2 = \sum_{j=0}^N (1 - j/N)^2 \| e_j\|_{A^2(\D,\mu)}^2 =  2\pi \sum_{j=0}^N (1 - j/N)^2  \sigma_j.
\]
Then, by \eqref{A=BH},
\[
\liminf_{N\to\infty} \|\mathcal{P}(\mathcal{F}_N)\|_{B^2(\D,\mu)+h^1(\D)}^2  \ge \liminf_{N\to\infty} \frac{\|\mathcal{P}(\mathcal{F}_N)\|_{A^2(\D,\mu)}^2}{C^2}  \ge \frac{2 \pi}{C^2} \sum_{j=0}^\infty \sigma_j  = \infty.
\]
This contradicts to the following inequality (which is a consequence of    \eqref{bdd-P} and \eqref{Fejer-norm})
\[
\sup_{N\ge 1} \| \mathcal{P}(\mathcal{F}_N)\|_{B^2(\D,\mu)+h^1(\D)}\le \| P\|<\infty.
\]
Hence we complete the whole proof of the  theorem.

\section{The upper half plane case}
In this section, we will prove Theorem \ref{thm-upper-stable}.
For any Radon measure $\Pi$ on $\R_{+} = (0, \infty)$ satisfying  the condition  \eqref{good-Pi},  define
\begin{align}\label{def-Lxi}
\mathcal{L}_\Pi(\xi):  =  \left\{\begin{array}{cl} \int_{\R^{+}}e^{-4 \pi y|\xi|}\Pi(dy)& \text{if $\xi \in \R^{*}=\R\setminus\{0\}$},
\\
0 & \text{if $\xi = 0$}.
\end{array}
\right.
\end{align}
We use the following definition of the Fourier transform for $f\in L^1(\R)$:
\[
\widehat{f}(\xi): = \int_\R f(x) e^{- i 2 \pi x \xi} dx, \quad \xi \in \R.
\]
For any $g\in \hb^2(\hh, \mu)$ and $y>0$, recall  the function  $g_y$ defined in \eqref{def-gline}, i.e.,
\begin{align*}
g_y(x) : =  g(x+iy), \, \forall \, x\in \R.
\end{align*}
Note that the Fourier transform of the Poisson kernel  $P_{iy}^\hh$ given in \eqref{H-Poi-kernel} has the following form (see \cite[Chapter VI, p. 140]{Katznelson}):
\[
\widehat{P^\hh_{iy}}(\xi)= e^{-2\pi y |\xi|}, \quad \xi \in \R.
\]
Then, by \eqref{def-pc} and  \eqref{hb-and-b}, we have  $
\widehat{g}_y(\xi)=  e^{-2\pi (y-y')|\xi|} \widehat{g}_{y'}(\xi)$ for all $0< y'<y$
and hence
\begin{align}\label{y-y-consistency}
e^{2\pi y|\xi|} \widehat{g}_y(\xi)=  e^{2\pi y'|\xi|} \widehat{g}_{y'}(\xi), \quad \forall \, 0< y'<y.
\end{align}

\begin{definition*}
Let $\mu(dz)= dx \Pi(dy)$ be a boundary-touching $(1,2)$-Carleson measure on $\hh$. For any $g\in \hb^2(\hh, \mu)$,   define  a function $\widehat{g}_0$ by
\begin{align}\label{def-uhat}
\widehat{g}_0(\xi) : =  e^{2\pi y|\xi|} \widehat{g}_y(\xi), \, y >0,
\end{align}
where,  by \eqref{y-y-consistency}, the right hand side of the equality \eqref{def-uhat} is independent of $y>0$.
\end{definition*}
By \eqref{def-uhat} and the Plancherel's identity, the norm  of any $g\in \hb^2(\hh, \mu)$ has the form:
\begin{align}\label{Bmu-norm-exp}
\|g\|_{\hb^2(\hh, \mu)}   = \Big(\int_\R  |\widehat{g}_0(\xi)|^2  \mathcal{L}_\Pi(\xi)d\xi\Big)^{1/2}.
\end{align}


\begin{remark*}
If the function $\widehat{g}_0$ defined in \eqref{def-uhat} belongs to $L^2(\R)$, then it is the Fourier transform of a function $g_0\in L^2(\R)$ and   the equality \eqref{def-uhat} is equivalent to
$
g_y = P^\hh_{iy}* g_0.
$
However,  the notation $\widehat{g}_0$ is only formal for a general $g\in \hb^2(\hh, \mu)$, that is,  it may not correspond to the Fourier transform of a generalized function $g_0$ on $\R$.
\end{remark*}

\begin{definition*}
Suppose that $\mu(dz)= dx \Pi(dy)$ is a boundary-touching $(1,2)$-Carleson measure on $\hh$. Let $H_\mu(\R)$ be the Hilbert space defined by  the norm completion as follows:
\[
H_\mu(\R) := \overline{\left\{   f \in L^2(\R)\Big|  \|f\|_{H_\mu(\R)} = \Big(
    \int_\R  |\widehat{f}(\xi)|^2  \mathcal{L}_\Pi(\xi) d\xi\Big)^{1/2}<\infty
\right\}}^{\| \cdot \|_{H_\mu(\R)}}.
\]
\end{definition*}
For any $f\in L^2(\R)$, set
\[
\mathcal{P}^\hh(f) (z): = (P_{iy}^\hh* f)(x), \quad z = x + iy \in \hh.
\]
Then for any $f\in L^2(\R)$ with
$
 \int_\R  |\widehat{f}(\xi)|^2  \mathcal{L}_\Pi(\xi) d\xi<\infty,
$
we have
\begin{align}\label{eqn-module-equal}
\left\|\mathcal{P}^\hh(f)\right\|_{\hb^2(\hh, \mu)}^2=\int_\R  |\widehat{f}(\xi)|^2  \mathcal{L}_\Pi(\xi) d\xi=\|f\|_{H_\mu(\R)}^2.
\end{align}


Similar to the disk case,  for a given measure $\mu(dz)= dx \Pi(dy)$ on $\hh$, a pair $(a,b)$ of two functions on $\R$ is called $\mu$-adapted if the following conditions are satisfied:
\begin{itemize}
\item[(i)] for any $\xi\in \R^*$,
\begin{align}\label{ab-prod}
|a(\xi)|^2 b(\xi) \sgn(\xi)  = \mathcal{L}_\Pi(\xi)^{-1};
\end{align}
\item[(ii)] there exists a constant $C_b>0$ such that
\begin{align}\label{def-cb-H}
\frac{1}{C_b}\leq |b(\xi)|\leq C_b.
\end{align}
\end{itemize}

Given any Radon measure $\Pi$ on $\mathbb{R}_+$ satisfying \eqref{good-Pi},  Garnett's result stated in Lemma~\ref{lem-Ramey} implies that the following  function $W^\Pi$ belongs to $L^\infty(\mathbb{R})$:
\begin{align}\label{def-W-pi}
W^\Pi(x):=i\int_{\R_{+}}\frac{\pi x}{y^2+\pi^2x^2}\Pi(dy), \quad x\in \R.
\end{align}

\begin{proposition}\label{prop-2B-ineq-H}
Suppose that  $\mu(dz)= dx \Pi(dy)$ is a boundary-touching $(1,2)$-Carleson measure on $\hh$ and let $(a,b)$ be  a $\mu$-adapted pair of functions defined on $\R$. Then  for any $u\in L^2(\R)$,
\begin{align}\label{2B-H-goal}
\|u \|_{L^2(\mathbb{\R})}\leq  C
 (\|\mathcal{T}_a u\|_{H_\mu(\mathbb{R})+L^1(\mathbb{R})}
+\|\mathcal{T}_a\mathcal{T}_b u\|_{H_\mu(\mathbb{R})+L^1(\R)}),
\end{align}
where
$\mathcal{T}_a, \mathcal{T}_b$ are the Fourier multipliers associated to $a, b$ given by
\[
\widehat{\mathcal{T}_a u}(\xi)=a(\xi)\widehat{u}(\xi),\quad \widehat{\mathcal{T}_b (u)}(\xi)=b(\xi)\widehat{u}(\xi)
\]
and the constant $C = C(b, \Pi)>0$ can be taken to be
\begin{align}\label{const-bpi}
 C(b, \Pi) = \sqrt{C_b+ \|W^\Pi\|_{L^{\infty}(\R)}+1} <\infty
\end{align}
\end{proposition}
\begin{remark*}
If  either $\mathcal{T}_a u$ or $\mathcal{T}_a \mathcal{T}_b u$ does not belong to $H_\mu(\R)+L^1(\R)$, then the right hand side of \eqref{2B-H-goal} is understood as $\infty$.
\end{remark*}

\subsection{The derivation of Theorem \ref{thm-upper-stable} from Proposition \ref{prop-2B-ineq-H}}
Let $\mu(dz) =dx\Pi(dy)$ be a boundary-touching  Radon measure on $\hh$.   If  \[
(\hb^2(\hh,\mu) + h^1(\hh)) \cap \Hol(\hh) = \hb^2(\hh,\mu) \cap \Hol(\hh),
\]
 then $H^1(\hh) = h^1(\hh) \cap \Hol(\hh) \subset \hb^2(\hh,\mu)$ and   by the Closed Graph Theorem, this embedding is continuous: there exists $C> 0$ such that
\begin{align}\label{H1-to-B2}
\| f\|_{\hb^2(\hh,\mu)}\le C \| f \|_{H^1(\hh)}, \quad \forall f \in H^1(\hh).
\end{align}
 Recall the definition \eqref{def-pc} of  the space $\pc(\hh)$. Since  $H^1(\hh)  \subset \pc(\hh)$,
\begin{align}\label{BH-L2}
\| f\|_{\hb^2(\hh,\mu)}^2 = \int_\hh | f(z)|^2\mu(dz), \quad \forall f\in H^1(\hh).
\end{align}
The inequality  \eqref{H1-to-B2} and the equality \eqref{BH-L2} together imply that  the measure $\mu$ is a $(1,2)$-Carleson measure.

Suppose now that $\mu(dz)= dx \Pi(dy)$ is a  boundary-touching $(1,2)$-Carleson measure on $\hh$.
Assume that
\begin{align}\label{ass-dec}
f = g + h \text{\, with \,} f \in \Hol(\hh),\quad g\in \hb^2(\hh, \mu),\an h\in h^1(\hh).
\end{align}
Then,  the goal is to show that $f \in \hb^2(\hh, \mu)$. It suffices to show
\begin{align}\label{f-less-gh}
\| f\|_{\hb^2(\hh, \mu)} \le 2 \sqrt{2 + \|W^\Pi\|_{L^{\infty}(\R)}} (\| g\|_{\hb^2(\hh, \mu)} + \| h\|_{h^1(\hh)}).
\end{align}
To avoid technical issues, we first consider the truncated measures of $\mu$. That is,   for any $R>0$, define
\[
\mu_R(dz)= dx \Pi_R(dy), \where \Pi_R(dy) = \mathds{1}(y<R) \cdot \Pi(dy).
\]
Define $W^{\Pi_R} \in L^\infty(\R)$ in  a similar way as in \eqref{def-W-pi}. Then  $\|W^{\Pi_R}\|_{L^{\infty}(\R)} \le \|W^\Pi\|_{L^{\infty}(\R)}$ for any $R>0$.  Therefore,  the desired inequality \eqref{f-less-gh} follows  from
\begin{align}\label{f-less-gh-trun}
\| f\|_{\hb^2(\hh, \mu_R)} \le 2 \sqrt{2 + \|W^{\Pi_R}\|_{L^{\infty}(\R)}} \Big(\| g\|_{\hb^2(\hh, \mu_R)} + \| h\|_{h^1(\hh)}\Big).
\end{align}

Now we are going to apply Proposition \ref{prop-2B-ineq-H}.  For any $R> 0$, define a  $\mu_R$-adapted pair $(a_R,b)$ of functions by
\[
a_R(\xi): =  \mathcal{L}_{\Pi_R}(\xi)^{-1/2}  = \Big(\int_0^R e^{-4 \pi y|\xi|} \Pi(dy)\Big)^{-1/2} \an b(\xi) = \sgn(\xi).
\]
In particular, by \eqref{good-Pi},
\begin{align}\label{bdd-aR}
    \sup_{\xi\in \R} a_R(\xi)^{-1} \le \sqrt{ \Pi((0, R))}<\infty.
\end{align}
  For any $y> 0$, define $f_y: \R\rightarrow \C$ and $f^y: \hh\rightarrow \C$  by
\[
f_y(x)= f(x+iy), \, x\in \R \an f^y(z) = f(z+ iy), \, z\in \hh.
\]
And  $g_y, g^y,  h_y, h^y$ are defined similarly.
{\flushleft \bf Claim I.} For any $\varepsilon> 0$, the function $f_\varepsilon$ belongs to the classical analytic Hardy space $H^2(\R)$ and hence
\begin{align}\label{supp-F-pos}
\supp(\widehat{f}_\varepsilon)\subset [0, \infty).
\end{align}
\begin{remark*}
The assertion \eqref{supp-F-pos} does not follow from the fact that $f_\varepsilon$ is the restriction  onto the real line of  a holomorphic function  defined on a neighborhood of the closed upper-half plane. For instance,  the following function
\[
K(x): = \frac{\sin(\pi x)}{\pi x}, \, x\in \R
\]
belongs to $L^2(\R)$ and is the restriction of an entire function on the complex plane. However, $\supp(\widehat{K}) = [-1/2, 1/2]\not\subset [0, \infty)$.
\end{remark*}

Since $h\in  h^1(\hh)$, by \cite[Chapter I, formula (3.9)]{garnett}, we have
\[
|h(z)|\leq \frac{2}{\pi y}\sup_{L>0}\int_{\R}|h(t+iL)|dt<\infty,\quad z=x+iy\in\hh.
\]
This implies that $h$ is uniformly bounded on $\hh_\varepsilon$ for any $\varepsilon>0$, where $\hh_\varepsilon$ is defined as \eqref{eqn-hh-varepsilon}. Thus
$
h_y\in L^1(\R)\cap L^\infty(\R) \subset L^2(\R).
$
By  \cite[Chapter I, Lemma 3.4]{garnett}, we have
\begin{align}\label{eqn-convolution-of-h}
h_y=P_{i(y-y')}^\hh*h_{y'},\quad \forall y>y'>0.
\end{align}
By \eqref{hb-and-b} (or equivalently, \eqref{Poi-cons} of Proposition \ref{prop-zen}), the assumption $g\in \hb^2(\hh, \mu)$ implies that  $g_y\in L^2(\R)\cap L^\infty(\R)$.  Consequently
\begin{align}\label{fy-L2}
f_y = g_y+ h_y\in L^2(\R).
\end{align}
Again by \eqref{hb-and-b} and \eqref{eqn-convolution-of-h},  for any $\varepsilon> 0$,
\[
f_y =  P_{i(y-\varepsilon)}^\hh *g_\varepsilon+  P_{i(y-\varepsilon)}^\hh *h_\varepsilon, \quad \forall  y \ge  \varepsilon.
\]
Therefore,  for any $\varepsilon>0$,
\begin{align*}
 \sup_{y>\varepsilon} \Big(\int_\R | f(x+iy)|^2 dx \Big)^{1/2} \le \| g_\varepsilon\|_{L^2(\R)} + \| h_\varepsilon\|_{L^2(\R)}    <\infty.
\end{align*}
The above inequality combined with the assumption $f\in \Hol(\hh)$ implies that $f_\varepsilon\in H^2(\R)$. This completes the proof of Claim I.

Now recall the decomposition \eqref{ass-dec}. Since $g \in \hb^2(\hh, \mu)$,  for any $\varepsilon>0$,
we have that $g^\varepsilon\in \hb^2(\hh, \mu)$ and hence  $g^\varepsilon \in \hb^2(\hh, \mu_R)$.  Since $g_\varepsilon\in L^2(\R)$ and
\[
g^\varepsilon(z)=\mathcal{P}^\hh(g_\varepsilon) (z)=(P_{iy}^{\hh}*g_\varepsilon)(x),\quad z=x+iy,
\]
by \eqref{eqn-module-equal}, we get  
\[
\| g_\varepsilon\|_{H_{\mu_R}(\R)} =  \| g^\varepsilon\|_{\hb^2(\hh, \mu_R)}.
\]
Note that the equality \eqref{fy-L2} and the inequality \eqref{bdd-aR} together imply that the function $a_R(\xi)^{-1} \widehat{f}_\varepsilon(\xi)$ belongs to $L^2(\R)$. Then  there exists a unique function  $u_\varepsilon\in L^2(\R)$ with
\begin{align}\label{u-epsilonhat}
\widehat{u}_\varepsilon(\xi) = a_R(\xi)^{-1} \widehat{f}_\varepsilon(\xi).
\end{align}
Hence, by \eqref{supp-F-pos}, $\supp(\widehat{u}_\varepsilon)\subset [0, \infty)$. It follows that,
\[
\widehat{f}_\varepsilon(\xi) = a_R(\xi) \widehat{u}_\varepsilon(\xi)    = a_R(\xi) \sgn (\xi) \widehat{u}_\varepsilon(\xi) =  a_R(\xi) b(\xi) \widehat{u}_\varepsilon(\xi).
\]
That is, $
f_\varepsilon = \mathcal{T}_{a_R} u_\varepsilon  = \mathcal{T}_{a_R} \mathcal{T}_{b}  u_\varepsilon. $
Therefore, since $u_\varepsilon\in L^2(\R)$, we may apply  \eqref{2B-H-goal} and get
\begin{align*}
\| u_\varepsilon\|_{L^2(\R)}\le   &  2 \sqrt{2 + \|W^{\Pi_R}\|_{L^{\infty}(\R)}} \| f_\varepsilon\|_{H_{\mu_R}(\R)+L^1(\R)}
\\
\le &   2 \sqrt{2 + \|W^{\Pi_R}\|_{L^{\infty}(\R)}}  \Big(\| g_\varepsilon\|_{H_{\mu_R}(\R)} +  \| h_\varepsilon\|_{L^1(\R)}\Big)
\\
= & 2 \sqrt{2 + \|W^{\Pi_R}\|_{L^{\infty}(\R)}}  \Big(\| g^\varepsilon\|_{\hb^2(\hh,\mu_R)} +  \| h_\varepsilon\|_{L^1(\R)}\Big)
\\
\le & 2 \sqrt{2 + \|W^{\Pi_R}\|_{L^{\infty}(\R)}}  \Big(\| g\|_{\hb^2(\hh,\mu_R)} +  \| h\|_{H^1(\hh)}\Big),
\end{align*}
where the last inequality is due to the  elementary observation:  for any $\varepsilon> 0$,
\[
 \| g^\varepsilon\|_{\hb^2(\hh, \mu_R)}\le  \| g\|_{\hb^2(\hh, \mu_R)} \an\quad     \| h_\varepsilon\|_{L^1(\R)} \le \| h\|_{h^1(\hh)}.
\]
Finally, by Plancherel's identity, the equalities \eqref{u-epsilonhat} and \eqref{Bmu-norm-exp},
\[
\| u_\varepsilon\|_{L^2(\R)}^2 = \int_\R \frac{|\widehat{f}_\varepsilon(\xi)|^2}{a_R(\xi)^2} d\xi  = \int_\R  |\widehat{f}_\varepsilon(\xi)|^2 \mathcal{L}_{\Pi_R}(\xi) d\xi   = \| f^\varepsilon\|_{\hb^2(\hh, \mu_R)}^2.
\]
Thus,
\[
\| f^\varepsilon\|_{\hb^2(\hh, \mu_R)} \le  2 \sqrt{2 + \|W^{\Pi_R}\|_{L^{\infty}(\R)}}  \Big( \| g\|_{\hb^2(\hh, \mu_R)}+  \| h\|_{h^1(\hh)}\Big).
\]
The  inequality  \eqref{f-less-gh-trun} now follows immediately since
\[
\lim_{\varepsilon\to 0^{+}} \| f^\varepsilon\|_{\hb^2(\hh, \mu_R)} = \| f\|_{\hb^2(\hh, \mu_R)}.
\]
\subsection{The proof of Proposition \ref{prop-2B-ineq-H}}

\begin{lemma}\label{lem-key-for-upper}
Let $\Pi$ be a Radon measure on $\mathbb{R}_+$ satisfying \eqref{good-Pi}.
Then  there is a function $W^\Pi \in L^\infty(\mathbb{R})$ such that  the following equality
\begin{align}\label{L1L2-test}
\int_{\mathbb{R}}u(x) \overline{W^\Pi(x)}dx=
\int_{\mathbb{R}}\widehat{u}(\xi)\sgn(\xi)
\mathcal{L}_\Pi(\xi)d\xi
\end{align}
holds for  all $u \in L^1(\R)\cap L^\infty(\R)$ satisfying
\begin{align}\label{u-hat-L1}
\int_\R |\widehat{u}(\xi)| \mathcal{L}_\Pi(\xi) d\xi <\infty.
\end{align}
\end{lemma}
\begin{remark*}
The equality \eqref{L1L2-test} means that the Fourier transform, in a certain distributional sense, of the function  $W^\Pi$,  is given by
$
\widehat{W^\Pi}(\xi)=\sgn(\xi)\mathcal{L}_\Pi(\xi)$.
If $\Pi(dy)=dy$ is the Lebesgue measure on $\R_{+}$, then
\[
  \mathcal{L}_\Pi(\xi)= \frac{1}{2|\xi|} \an W^\Pi(x)= \frac{i\pi}{2}\sgn(x).
\]
In general,  the Fourier transform of  $W^\Pi$ can only be understood in a certain distributional sense and the condition \eqref{u-hat-L1} in Lemma \ref{lem-key-for-upper} can not be removed.
\end{remark*}

The proof of Lemma \ref{lem-key-for-upper} is postponed to the end of this section.

\begin{proof}[Proof of Proposition \ref{prop-2B-ineq-H}]
Take  $u\in L^2(\R)$. Suppose that we have  decompositions
\begin{align}\label{Ta-Tab-dec}
\mathcal{T}_a u(x)=f_1(x)+g_1(x) \an \mathcal{T}_a \mathcal{T}_b u(x)=f_2(x)+g_2(x)
\end{align}
with $f_1,f_2\in H_\mu(\R)$ and $g_1,g_2\in L^1(\R)$. That is,
\begin{align}\label{2-dec-H}
\widehat{u}(\xi)a(\xi)=\widehat{f}_1(\xi)+\widehat{g}_1(\xi),\quad \widehat{u}(\xi)a(\xi)b(\xi)=\widehat{f}_2(\xi)+\widehat{g}_2(\xi).
\end{align}
For any fixed $y>0$, applying the Poisson convolution to both sides of \eqref{Ta-Tab-dec}, we have
\[
P_{iy}^\hh*\mathcal{T}_a u=P_{iy}^\hh*f_1+P_{iy}^\hh*g_1,\quad P_{iy}^\hh*(\mathcal{T}_a\mathcal{T}_b )u=P_{iy}^\hh*f_2+P_{iy}^\hh*g_2.
\]
By Plancherel's identity and   \eqref{2-dec-H},
\begin{align}\label{I1+I2}
\begin{split}
\|P_{iy}^\hh*u\|^2_{L^2(\R)}=&\int_\R |\widehat{P_{iy}^\hh}(\xi)|^2|\widehat{u}(\xi)|^2d\xi=\int_\R |\widehat{P_{iy}^\hh}(\xi)|^2\widehat{u}(\xi)\overline{\widehat{u}(\xi)}d\xi\\
=&\int_\R |\widehat{P_{iy}^\hh}(\xi)|^2\frac{\widehat{f}_1(\xi)+\widehat{g}_1(\xi)}{a(\xi)}\overline{\widehat{u}(\xi)} d\xi\\
=& \underbrace{\int_\R |\widehat{P_{iy}^\hh}(\xi)|^2\frac{\widehat{f}_1(\xi)}{a(\xi)}\overline{\widehat{u}(\xi)}d\xi}_{\text{denoted by $\mathrm{I}_1$}}+\underbrace{\int_\R |\widehat{P_{iy}^\hh}(\xi)|^2\frac{\widehat{g}_1(\xi)}{a(\xi)}\overline{\widehat{u}(\xi)}d\xi}_{\text{denoted by $\mathrm{I}_2$}}.
\end{split}
\end{align}
Cauchy-Schwarz's inequality and the conditions \eqref{ab-prod}, \eqref{def-cb-H} together imply
\begin{align}\label{I-1-es}
\begin{split}
|\mathrm{I}_1|& \leq \Big(\int_\R |\widehat{P_{iy}^\hh}(\xi)\widehat{u}(\xi)|^2\Big)^{1/2}\Big(\int_\R \Big|\frac{\widehat{P_{iy}^\hh}(\xi)\widehat{f}_1(\xi)}{a(\xi)}\Big|^2\Big)^{1/2}
\\
&\leq \sqrt{C_b}\left\|P_{iy}^\hh*u\right\|_{L^2(\R)}\left\|P_{iy}^\hh*f_1\right\|_{H_\mu(\R)}
\\
& \le \sqrt{C_b}\left\|P_{iy}^\hh*u\right\|_{L^2(\R)}\|f_1\|_{H_\mu(\R)}.
\end{split}
\end{align}
And, by \eqref{2-dec-H} and $b(\xi)\in \R$, the integral $\mathrm{I}_2$ can be decomposed as
\begin{align}\label{I2=34}
\begin{split}
\mathrm{I}_2=&\int_\R |\widehat{P_{iy}^\hh}(\xi)|^2\frac{\widehat{g}_1(\xi)}{a(\xi)} \overline{\Big(\frac{\widehat{f_2}(\xi)+\widehat{g}_2(\xi)}{a(\xi)b(\xi)}\Big)}d\xi\\
=&\underbrace{\int_\R |\widehat{P_{iy}^\hh}(\xi)|^2\frac{(a(\xi)\widehat{u}(\xi)-\widehat{f}_1(\xi)) \overline{\widehat{f}_2(\xi)}}{|a(\xi)|^2 b(\xi)}d\xi}_{\text{denoted by $\mathrm{I}_3$}}+ \underbrace{\int_\R |\widehat{P_{iy}^\hh}(\xi)|^2\frac{\widehat{g}_1(\xi)\overline{\widehat{g}_2(\xi)}}{|a(\xi)|^2 b(\xi)}d\xi}_{\text{denoted by $\mathrm{I}_4$}}.
\end{split}
\end{align}
The integral $\mathrm{I}_3$ can be easily controlled. Indeed,  again by  Cauchy-Schwarz's inequality, \eqref{ab-prod} and \eqref{def-cb-H},
\begin{align}\label{I-3-es}
\begin{split}
|\mathrm{I}_3|\leq& \Big(\int_\R |\widehat{P_{iy}^\hh}(\xi)\widehat{u}(\xi)|^2 d\xi \Big)^{1/2}\Big(\int_\R \Big|\frac{\widehat{P_{iy}^\hh}(\xi)\widehat{f}_2(\xi)}{a(\xi)b(\xi)}\Big|^2 d\xi\Big)^{1/2}
\\
& + \Big(\int_\R \frac{|\widehat{P_{iy}^\hh}(\xi)\widehat{f}_1(\xi)|^2}{|a(\xi)|^2 |b(\xi)|} d\xi \Big)^{1/2}\Big(\int_\R \frac{|\widehat{P_{iy}^\hh}(\xi)\widehat{f}_2(\xi)|^2}{|a(\xi)|^2 |b(\xi)|} d\xi\Big)^{1/2}\\
\leq& \sqrt{C_b} \|P_{iy}^\hh*u\|_{L^2(\R)}\|P_{iy}^\hh*f_2\|_{H_\mu(\R)}+\|P_{iy}^\hh*f_1\|_{H_\mu(\R)}\|P_{iy}^\hh*f_2\|_{H_\mu(\R)}
\\
\le & \sqrt{C_b} \|P_{iy}^\hh*u\|_{L^2(\R)}\|f_2\|_{H_\mu(\R)}+\|f_1\|_{H_\mu(\R)}\|f_2\|_{H_\mu(\R)}.
\end{split}
\end{align}
It remains to estimate the integral  $\mathrm{I}_4$.  Since $g_1, g_2\in L^1(\R)$, for any $y>0$,  one can define
\begin{align}\label{def-Gy}
G_y:=(P_{iy}^\hh*g_1)*(P_{iy}^\hh *\tilde{g}_2)=P_{2iy}^\hh*(g_1*\tilde{g}_2), \where \tilde{g}_2(x):=g_2(-x).
\end{align}
In particular,
\begin{align}\label{Gy-Fourier}
\widehat{G_y}(\xi) = |\widehat{P_{iy}^\hh}(\xi)|^2 \widehat{g}_1(\xi) \overline{\widehat{g}_2(\xi)} = e^{-4\pi y|\xi|} \widehat{g}_1(\xi) \overline{\widehat{g}_2(\xi)}.
\end{align}

 {\flushleft \bf Claim A.} For any $y>0$, the function $G_y$ defined in \eqref{def-Gy} satisfies
\begin{align}\label{Gy-L1Linf}
G_y\in L^1(\R)\cap L^\infty(\R)
\end{align}
and
\begin{align} \label{Gy-hat}
\int_\R |\widehat{G_y}(\xi)|\mathcal{L}_\Pi(\xi)d\xi<\infty.
\end{align}

Indeed, $g_1*\tilde{g}_2\in L^1(\R)$ since $g_1, g_2\in L^1(\R)$.  Therefore, \eqref{Gy-L1Linf} follows from the definition  \eqref{def-Gy} and  the  observation that $P_{2iy}^\hh\in L^1(\R)\cap L^\infty(\R)$. By \eqref{2-dec-H},
\[
\frac{\widehat{g}_1(\xi)}{a(\xi)} = \widehat{u}(\xi) - \frac{\widehat{f}_1(\xi)}{a(\xi)}, \quad \frac{\widehat{g}_2(\xi)}{a(\xi) b(\xi)} = \widehat{u}(\xi) - \frac{\widehat{f}_2(\xi)}{a(\xi)b(\xi)}.
\]
The assumptions $f_1, f_2\in H_\mu(\R)$ combined with the conditions  \eqref{ab-prod}, \eqref{def-cb-H} on  the pair $(a,b)$ imply that both functions $\widehat{f}_1/a$ and $\widehat{f}_2/(ab)$ belong to $L^2(\R)$. Since $u\in L^2(\R)$ and hence $\widehat{u}\in L^2(\R)$, we obtain, by using \eqref{ab-prod} again, that
\begin{align*}
\int_\R |\widehat{G_y}(\xi)| \mathcal{L}_\Pi(\xi)d\xi  &   = \int_\R e^{-4 \pi y|\xi|}  \frac{|\widehat{g}_1(\xi)|}{|a(\xi)|} \cdot \frac{|\widehat{g}_2(\xi)|}{|a(\xi)b(\xi)|} d\xi  \\
&   \le  \int_\R  \Big|\widehat{u}(\xi) - \frac{\widehat{f}_1(\xi)}{a(\xi)}\Big|  \cdot \Big |\widehat{u}(\xi) - \frac{\widehat{f}_2(\xi)}{a(\xi)b(\xi)} \Big| d\xi
\\
& \le \Big\| \widehat{u} - \frac{\widehat{f}_1}{a}\Big\|_{L^2(\R)} \Big\| \widehat{u} - \frac{\widehat{f}_2}{ab}\Big\|_{L^2(\R)} <\infty.
\end{align*}
By Claim A, the function $G_y$ satisfies all the required conditions of Lemma \ref{lem-key-for-upper}. Hence, by \eqref{Gy-Fourier}, \eqref{ab-prod} and \eqref{L1L2-test},
\[
\mathrm{I}_4 = \int_\R  \widehat{G_y}(\xi) \sgn(\xi) \mathcal{L}_\Pi(\xi) d\xi= \int_{\mathbb{R}}G_y(x) \overline{W^\Pi(x)}dx.
\]
It follows that
\begin{align}\label{I-4-es}
\begin{split}
|\mathrm{I}_4|& \le \|W^\Pi\|_{L^{\infty}(\R)} \|G_y\|_{L^1(\R)}   =  \|W^\Pi\|_{L^{\infty}(\R)} \|P_{2iy}^\hh*(g_1*\tilde{g}_2)\|_{L^1(\R)}
\\
& \le \|W^\Pi\|_{L^{\infty}(\R)} \| g_1\|_{L^1(\R)}\|g_2\|_{L^1(\R)}.
\end{split}
\end{align}
Combining \eqref{I1+I2}, \eqref{I-1-es}, \eqref{I2=34}, \eqref{I-3-es} and \eqref{I-4-es}, we get
\begin{align*}
\|P_{iy}^\hh*u\|^2_{L^2(\R)}\le &    \sqrt{C_b}\left\|P_{iy}^\hh*u\right\|_{L^2(\R)}\|f_1\|_{H_\mu(\R)}   + \sqrt{C_b} \|P_{iy}^\hh*u\|_{L^2(\R)}\|f_2\|_{H_\mu(\R)}
\\
 & +\|f_1\|_{H_\mu(\R)}\|f_2\|_{H_\mu(\R)}  + \|W^\Pi\|_{L^{\infty}(\R)} \| g_1\|_{L^1(\R)}\|g_2\|_{L^1(\R)}.
\end{align*}
Therefore, by a standard argument, there exists a constant $C>0$ depending only on the constants $C_b$ and $\|W^\Pi\|_{L^{\infty}(\R)}$ such that
\[
\|P_{iy}^\hh*u\|_{L^2(\R)}\le  C ( \| f_1\|_{H_\mu(\R)} + \| g_1\|_{L^1(\R)}  +  \| f_2\|_{H_\mu(\R)} + \| g_2\|_{L^1(\R)} ).
\]
The constant $C$ in the above inequality can be taken to be
\[
C =  \sqrt{C_b+ \|W^\Pi\|_{L^{\infty}(\R)}+1}.
\]
Since the decompositions \eqref{Ta-Tab-dec} are arbitrary, we get
\[
\|P_{iy}^\hh*u\|_{L^2(\R)}\le  C ( \| \mathcal{T}_a u \|_{H_\mu(\R) + L^1(\R)}  +  \| \mathcal{T}_b \mathcal{T}_a u  \|_{H_\mu(\R) +L^1(\R)} ).
\]
Finally, by taking the limit $y\to 0^{+}$ and using
\[
\lim_{y\to 0^{+}} \|P_{iy}^\hh*u\|_{L^2(\R)} = \|u\|_{L^2(\R)},
\]
we obtain the desired inequality \eqref{2B-H-goal} and complete the whole proof of the proposition.
\end{proof}

\begin{proof}[Proof of Lemma \ref{lem-key-for-upper}]
Fix  a Radon measure $\Pi$ on $\mathbb{R}_+$ satisfying \eqref{good-Pi}. By  Garnett's result stated in Lemma \ref{lem-Ramey}, one can define a function $W^\Pi\in L^\infty(\mathbb{R})$ by \eqref{def-W-pi}.

Now we show that $W^\Pi$ satisfies the equality \eqref{L1L2-test}.  For any $0<\varepsilon< R<\infty$, set
\begin{align}\label{def-trun-W}
W^\Pi_{\varepsilon, R}(x): = i\int_{\R_{+}} \frac{\pi x}{y^2+\pi^2x^2}\Pi_{\varepsilon, R}(dy),\where \Pi_{\varepsilon, R}(dy) = \mathds{1}(\varepsilon<y<R)\cdot\Pi(dy).
\end{align}

 {\flushleft \bf Claim B.}  For any $0<\varepsilon< R<\infty$, we have  $W^\Pi_{\varepsilon, R}\in L^2(\R)$ and  the Fourier transform of $W^\Pi_{\varepsilon, R}$ is given by the Bochner integral for  $L^2(\R)$-vector valued function:
\begin{align}\label{Bochner-W-F}
\widehat{W^\Pi_{\varepsilon, R}}=  \int_\varepsilon^R   \ell_y \Pi(dy),  \quad \ell_y(\xi):= \sgn(\xi) e^{-2y|\xi|}.
\end{align}
In particular, $\widehat{W^\Pi_{\varepsilon, R}}$ can be identified with a $C^\infty(\R^{*})$-function by the formula
\begin{align}\label{trun-W-F}
\widehat{W^\Pi_{\varepsilon, R}}(\xi)=  \int_\varepsilon^R \mathrm{sgn}(\xi)e^{-2y|\xi|}\Pi(dy),  \quad \xi \in \R^{*} = \R\setminus \{0\}.
\end{align}

Indeed, for any $y>0$, recall that the conjugate Poisson kernel (see, e.g., \cite[formula (4.1.16)]{Grafakos}) of $\hh$ is given by
\[
Q_{iy}^\hh(x)=\frac{\pi x}{y^2 + \pi^2 x^2}, \quad x\in\R.
\]
Clearly, $Q_{iy}^\hh\in L^2(\R)$ for all $y>0$ and  the map $y\mapsto Q_{iy}^\hh$
 is continuous from $\R_+$ to $L^2(\R)$, hence it is uniformly continuous from $[\varepsilon, R]$ to $L^2(\R)$.  Consequently, using  the definition \eqref{def-trun-W} of $W_{\varepsilon, R}^\Pi$ and the fact that $\Pi_{\varepsilon, R}$ is  a  finite measure with support contained in $[\varepsilon, R]$, we obtain  that $W_{\varepsilon, R}^\Pi\in L^2(\R)$ and  the following equality in the sense of the Bochner integral for  $L^2(\R)$-vector valued functions:
\[
\widehat{W_{\varepsilon, R}^\Pi}= i \int_\varepsilon^R    \widehat{Q_{iy}^\hh} \Pi(dy).
\]
Then the equality \eqref{Bochner-W-F} follows immediately since (see, e.g., \cite[formula (4.1.33)]{Grafakos})
\[
 \widehat{Q_{iy}^\hh}(\xi) = - i \ell_y (\xi) = -i \sgn(\xi) e^{-2 y |\xi|}.
\]

{\flushleft \bf Claim C.} For any $\varphi \in L^1(\R)$,
\begin{align}\label{test-phi-epsilon}
\lim_{\varepsilon\rightarrow 0^+}\int_{\R} \varphi(x)\Big[\int_0^\varepsilon \frac{\pi x}{y^2+\pi^2x^2}\Pi(dy)\Big]dx=0.
\end{align}

Indeed, for any $y >0$, set $
F_\Pi(y): = \Pi((0,y])$. Then
$\Pi(dy) = dF_\Pi(y)$ and by the assumption \eqref{good-Pi}, there exists a constant $C>0$ such that
\begin{align}\label{F-Pi-small}
F_\Pi(y) \leq Cy, \quad \forall  y>0.
\end{align}
 By integration by parts for the absolutely continuous function $F_\Pi$,
\begin{align}\label{Pi-FPi}
\int_0^\varepsilon\frac{\pi x}{y^2+\pi^2 x^2}\Pi(dy)=\frac{\pi x}{y^2+\pi^2 x^2} F_\Pi(y)\Big|_{y=0}^{y=\varepsilon}+\int_0^\varepsilon F_\Pi(y)\frac{2\pi xy}{(y^2+\pi^2 x^2)^2}dy.
\end{align}
In particular, if $\Pi(dy)$ is the Lebesgue measure on $\R_{+}$,  then the equality \eqref{Pi-FPi} becomes
\begin{align}\label{Pi-FPi-bis}
\arctan\left(\frac{\varepsilon}{\pi x}\right) = \int_0^\varepsilon \frac{\pi x}{y^2+\pi^2 x^2}dy  = \frac{\pi x}{y^2+\pi^2 x^2} y\Big|_{y=0}^{y=\varepsilon}+\int_0^\varepsilon y\frac{2\pi xy}{(y^2+\pi^2 x^2)^2}dy.
\end{align}
Comparing \eqref{Pi-FPi} and \eqref{Pi-FPi-bis} and using \eqref{F-Pi-small}, we obtain
\[
\int_0^\varepsilon\frac{\pi |x|}{y^2+\pi^2 x^2}\Pi(dy)\le C \arctan\left(\frac{\varepsilon}{\pi |x|}\right).
\]
Therefore,   by dominated convergence theorem, for any $\varphi\in L^1(\R)$,
\begin{align*}
\limsup_{\varepsilon\to 0^{+}}\Big| \int_{\R} \varphi(x)\int_0^\varepsilon \frac{\pi x}{y^2+\pi^2x^2}\Pi(dy)dx\Big|& \leq C \limsup_{\varepsilon\to 0^{+}}\int_{\R} |\varphi(x)|\arctan\left(\frac{\varepsilon}{\pi |x|}\right)dx=0.
\end{align*}

{\flushleft \bf Claim D.} For any $\varphi \in L^1(\R)$,
\begin{align}\label{Eqn-infinity-positive-negative-part}
\lim_{R\rightarrow \infty}\int_\R \varphi(x)\Big[\int_R^\infty \frac{\pi x}{y^2+\pi^2x^2}\Pi(dy)\Big] dx=0.
\end{align}
The proof of the equality \eqref{Eqn-infinity-positive-negative-part} is similar to that of \eqref{test-phi-epsilon} and thus is omitted here.

Now fix any $u\in L^1(\R) \cap L^\infty(\R)$ satisfying \eqref{u-hat-L1}. By \eqref{test-phi-epsilon} and \eqref{Eqn-infinity-positive-negative-part},
\begin{align}\label{E:key-eq-01}
\lim_{\substack{\varepsilon\to 0^+ \\ R\to \infty}}\int_\mathbb{R} u(x) \overline{W^\Pi_{\varepsilon,R}(x)}dx=
\int_{\mathbb{R}}u(x)\overline{W^\Pi(x)}dx.
\end{align}
Moreover, since both $u$ and $W^\Pi$ belong to $L^2(\R)$,  the Plancherel's identity implies
\begin{align}\label{Planch-id}
\int_\mathbb{R} u(x) \overline{W^\Pi_{\varepsilon,R}(x)}dx=
\int_\mathbb{R}\widehat{u}(\xi)
\overline{\widehat{W_{\varepsilon,R}^\Pi}(\xi)}d\xi  = \int_\R \widehat{u}(\xi) \sgn(\xi) \Big[ \int_\varepsilon^R e^{-2y|\xi|} \Pi(dy) \Big] d\xi.
\end{align}
Using the assumption \eqref{u-hat-L1}, we obtain, by dominated convergence theorem,
\begin{align}\label{DCT-trun}
\lim_{\substack{\varepsilon\to 0^{+} \\  R\to\infty}}  \int_\R \widehat{u}(\xi) \sgn(\xi) \Big[ \int_\varepsilon^R e^{-2y|\xi|} \Pi(dy) \Big] d\xi = \int_\R \widehat{u}(\xi) \sgn(\xi)  \mathcal{L}_\Pi(\xi) d\xi.
\end{align}
Combining \eqref{E:key-eq-01}, \eqref{Planch-id} and \eqref{DCT-trun}, we obtain the desired equality \eqref{L1L2-test}.
\end{proof}
\appendix
\section{Characterization of the Zen space}
Given a  horizontal translation-invariant  boundary-touching measure $\mu$  on $\hh$, recall the definition of  the harmonic Zen-type space in \eqref{def-hb}:
\[
\hb^2(\hh, \mu): = \Big\{g \in \Harm(\hh)\Big| \| g\|_{\hb^2(\hh, \mu)}  = \sup_{L>0} \Big(\int_\hh | g(z+iL)|^2 \mu(dz)\Big)^{1/2} <\infty \Big\}.
\]
In Proposition \ref{prop-zen}, we will give a characterization  of $\hb^2(\hh, \mu)$.

For any $g \in \Harm(\hh)$ and any $w\in \hh$,  define
\begin{align}\label{def-uwy}
g^w(z):  = g(z + w), \, z\in \hh
\end{align}

\begin{proposition}\label{prop-zen}
Let $\mu$ be a  horizontal translation-invariant boundary-touching measure on $\hh$. Then for any   $g\in \Harm(\hh)$, the following assertions are equivalent:
\begin{itemize}
\item[(1)] $g\in \hb^2(\hh, \mu)$;
\item[(2)] for any $\varepsilon>0$, the following inequality holds
\begin{align}\label{bdd-Hepsilon}
\sup_{L>\varepsilon} \int_\hh | g(z+iL)|^2   \mu(dz) <\infty;
\end{align}
\item[(3)] $g\in B^2(\hh, \mu)$ and the functions $g_y$ defined in \eqref{def-gline} satisfy
\begin{align}\label{Poi-cons}
g_y \in L^2(\R) \an g_y = P_{i(y-y')}^\hh * g_{y'}  \quad \text{for all $y>y'>0$,}
\end{align}
where $P_{z}^{\mathbb{H}}(t)$ is the Poisson kernel in \eqref{H-Poi-kernel}.
\item[(4)] the following inequality holds:
\begin{align}\label{zen-norm}
\sup_{L>0}\int_\hh |g(z+iL)|^2 \mu(dz)\le  \int_\hh |g(z)|^2 \mu(dz)<\infty.
\end{align}
\end{itemize}
\end{proposition}
\begin{proof}

 $(1)\Rightarrow (2)$ is direct from the definition.

 $(2)\Rightarrow (3)$.
 By Fatou's lemma and the assumption \eqref{bdd-Hepsilon}, it follows that
\[
\int_{\hh} |g(z)|^2\mu(dz)=\int_{\hh} \lim_{L\rightarrow 0}|g(z+iL)|^2\mu(dz)\leq \liminf_{L\rightarrow 0} \int_{\hh} |g(z+iL)|^2\mu(dz)<\infty.
\]
Hence $g\in B^2(\hh, \mu)$.  Using  \eqref{bdd-Hepsilon} again, we have
 \begin{align*}
 \int_{\hh} |g(z+iL)|^2\mu(dz)=\int_{0}^{\infty}\int_\R |g(x+iy+iL)|^2dx\Pi(dy)<\infty,\quad \forall L>0.
 \end{align*}
This implies that $g_{y+L}\in L^2(\R)$ for $\Pi$-almost  $y\in\R_+$. Since $\mu=dx\Pi(dy)$ is boundary accessible and $L>0$ is arbitrary, we get that for any $\eta>0$, there exists $y\in (0,\eta)$ such that $g_y\in L^2(\R)$.

For any fixed $\epsilon>0$, take $w=s+it$ with $\Im (w)=t\geq \epsilon$. We have
 \begin{align*}
\int_\hh|g^w(z)|^2\mu(dz)=&\int_{0}^{\infty}\int_\R |g(x+s+i(y+t)|^2dx\Pi(dy)\\
=&\int_{0}^{\infty}\int_\R |g(x+i(y+t))|^2dx\Pi(dy)\\
=&\int_\hh |g(z+it)|^2\mu(dz)\\
\leq &\sup_{t\geq\varepsilon} \int_\hh | g(z+it)|^2   \mu(dz) <\infty.
 \end{align*}
Hence the vector-valued harmonic function $G(w)$ defined by
\begin{align}
\hh\longrightarrow B^2(\hh,\mu):\quad w\in\hh \longrightarrow G(w):=g^w\in B^2(\hh,\mu)
\end{align}
satisfies
\[
\sup_{\Im(w)\geq \varepsilon }\|g^w\|_{B^2(\hh,\mu)}<\infty,\quad \forall \varepsilon >0.
\]
In other words, $G(w)$ is uniformly bounded on $\hh_\varepsilon$ for any $\varepsilon>0$.
Therefore, the function
\begin{align}
w\in\hh\longrightarrow \left\|G(w)\right\|_{B^2(\hh,\mu)}
\end{align}
is a subharmonic function on $\hh$ and uniformly bounded on $\hh_\varepsilon$ for any $\varepsilon>0$. For $y>0$, define
\begin{align*}
G_y(x):=G(x+iy),\quad x\in\R.
\end{align*}
By Maximal principle for vector-valued harmonic functions,
\begin{align}\label{E-G-y-yprime}
P_{i(\tilde{y}-\tilde{y}')}^\hh*G_{\tilde{y}'}(x)=G_{\tilde{y}}(x),\quad \forall \tilde{y}>\tilde{y}'>0
\end{align}
By \eqref{E-G-y-yprime}, we have
\begin{align*}
\int_\R P_{i(\tilde{y}-\tilde{y}')}^\hh(x-t)G_{\tilde{y}'}(t)dt=G_{\tilde{y}}(x),
\end{align*}
which implies
\begin{align*}
\int_\R P_{i(\tilde{y}-\tilde{y}')}^\hh(x-t)g(t+i\tilde{y}'+w)dt=g(x+i\tilde{y}+w),\quad \forall w\in\hh.
\end{align*}
In particular, take $w=i\delta$ with $\delta>0$, one has
\[
\int_\R P_{i(\tilde{y}-\tilde{y}')}^\hh(x-t)g(t+i(\tilde{y}'+\delta))dt=g(x+i(\tilde{y}+\delta)),\quad \forall \tilde{y}>\tilde{y}'>0.
\]
Take $y=\tilde{y}+\delta$ and $y'=\tilde{y}'+\delta$, we have
\begin{align}
\int_\R P_{i(y-y')}^\hh(x-t)g(t+iy')dt=g(x+iy),\quad \forall y>y'>0,
\end{align}
which is equivalent to
\begin{align}\label{E-g-y-yprime}
P_{i(y-y')}^\hh*g_{y'}=g_{y},\quad \forall y>y'>0.
\end{align}
Since for any $y>0$, one can always find $y'<y$ such that $g_{y'}\in L^2(\R)$, hence $g_y\in L^2(\R)$ by \eqref{E-g-y-yprime} and the contraction of Poisson convolution.

$(3)\Rightarrow (4)$ By \eqref{Poi-cons} and Plancherel's identity, we have
\begin{align*}
\widehat{g}_y(\xi)= e^{-2\pi (y-y')|\xi|}\widehat{g}_{y'}(\xi),\quad \forall y>y'>0.
\end{align*}
Hence we can define  a function $\widehat{g}_0$ by
\begin{align*}
\widehat{g}_0(\xi):=e^{2\pi y|\xi|} \widehat{g}_y(\xi), \, y >0,
\end{align*}
such that
\begin{align*}
\widehat{g}_y(\xi):=e^{-2\pi y|\xi|} \widehat{g}_0(\xi), \, y >0.
\end{align*}
Hence for any $L>0$, by Plancherel's identity, we get
\begin{align*}
\int_{\hh}|g(z+iL)|^2\mu(dz)=&\int_{0}^{\infty}\int_\R |g(x+iy+iL)|^2dx\Pi(dy)\\
=&\int_{0}^{\infty}\int_\R|g_{y+L}(x)|^2dx\Pi(dy)\\
=&\int_{0}^{\infty}\int_\R|\widehat{g}_{y+L}(\xi)|^2d\xi\Pi(dy)\\
=&\int_{0}^{\infty}\int_\R e^{-4\pi (y+L)|\xi|}|\widehat{g}_{0}(\xi)|^2d\xi\Pi(dy)\\
\leq & \int_{0}^{\infty}\int_\R e^{-4\pi y|\xi|}|\widehat{g}_{0}(\xi)|^2d\xi\Pi(dy).
\end{align*}
Similarly, one has
\begin{align*}
\int_{\hh}|g(z)|^2\mu(dz)=&\int_{0}^{\infty}\int_\R |g(x+iy)|^2dx\Pi(dy)\\
=&\int_{0}^{\infty}\int_\R|g_{y}(x)|^2dx\Pi(dy)\\
=&\int_{0}^{\infty}\int_\R|\widehat{g}_{y}(\xi)|^2d\xi\Pi(dy)\\
=&\int_{0}^{\infty}\int_\R e^{-4\pi y|\xi|}|\widehat{g}_{0}(\xi)|^2d\xi\Pi(dy).
\end{align*}
Hence the above equations imply the following
\[
\sup_{L>0}\int_\hh |g(z+iL)|^2 \mu(dz)\le  \int_\hh |g(z)|^2 \mu(dz)<\infty.
\]

$(4)\Rightarrow (1)$ is direct.
\end{proof}



\medskip

{\flushleft Conflict of interest statements:  The author declares no conflict of interest.}

\end{document}